\documentclass{amsproc}

\usepackage{verbatim}
\usepackage{xcolor}
\usepackage{tikz}
\usepackage{tikz-cd}
\usepackage{arydshln}
\usepackage{caption}
\usepackage{subcaption}
\usepackage{graphicx}
\tikzset{
	symbol/.style={
		draw=none,
		every to/.append style={
			edge node={node [sloped, allow upside down, auto=false]{$#1$}}}
	}
}

\usepackage{amsmath}
\usepackage{bbm}
\usepackage{hyperref}

\newtheorem{theorem}{Theorem}[section]
\newtheorem{proposition}[theorem]{Proposition}
\newtheorem{lemma}[theorem]{Lemma}
\newtheorem{corollary}[theorem]{Corollary}

\theoremstyle{definition}

\newtheorem{example}[theorem]{Example}

\theoremstyle{remark}
\newtheorem{remark}[theorem]{Remark}

\numberwithin{equation}{section}

\begin{document}

\title[Counting hyperbolic multi-geodesics]{Counting hyperbolic multi-geodesics with respect to the lengths of individual components}



\author{Francisco Arana--Herrera}

\email{farana@stanford.edu}

\address{Department of Mathematics, Stanford University, 450 Jane Stanford Way,
	Building 380, Stanford, CA 94305-2125, USA}



\date{}

\begin{abstract}
Given a connected, oriented, complete, finite area hyperbolic surface $X$ of genus $g$ with $n$ punctures, Mirzakhani showed that the number of multi-geodesics on $X$ of total hyperbolic length $\leq L$ in the mapping class group orbit of a given simple or filling closed multi-curve is asymptotic as $L \to \infty$ to a polynomial in $L$ of degree $6g-6+2n$. We establish asymptotics of the same kind for countings of multi-geodesics in mapping class group orbits of simple or filling closed multi-curves that keep track of the hyperbolic lengths of individual components, proving and generalizing a conjecture of Wolpert. In the simple case we consider more precise countings that also keep track of the class of the multi-geodesics in the space of projective measured geodesic laminations. We provide a unified geometric and topological description of the leading terms of the asymptotics of all the countings considered. Our proofs combine techniques and results from several papers of Mirzakhani as well as ideas introduced by Margulis in his thesis.
\end{abstract}

\maketitle


\thispagestyle{empty}

\tableofcontents

\section{Introduction}

$ $

Let $X$ be a connected, oriented, complete, finite area hyperbolic surface of genus $g$ with $n$ punctures. In \cite{Mir08b}, Mirzakhani showed that the number of rational simple closed multi-geodesics on $X$ of total hyperbolic length $\leq L$ in a given mapping class group orbit is asymptotic as $L \to \infty$ to a polynomial in $L$ of degree $6g-6+2n$. In \cite{Mir16}, using completely different methods, Mirzakhani established analogous counting results for filling closed hyperbolic multi-geodesics.\\

The main goal of this paper is to extend Mirzakhani's results in both the simple and filling cases to countings of closed hyperbolic multi-geodesics that keep track of the lengths of individual components. More precisely, given a connected, oriented, complete, finite area hyperbolic surface $X$ of genus $g$ with $n$ punctures, an ordered closed multi-curve $\gamma := (\gamma_1,\dots,\gamma_k)$ on $X$ with $k \geq 1$ components, and a vector $\mathbf{b} := (b_1,\dots,b_k) \in (\mathbf{R}_{>0})^k$, the aim of this paper is to understand the asymptotics as $L \to \infty$ of the counting function
\[
c(X,\gamma,\mathbf{b},L) := \# \{\alpha :=(\alpha_1,\dots,\alpha_k) \in \text{Mod}(X) \cdot \gamma \ | \ \ell_{\alpha_i}(X) \leq b_i L, \ \forall i =1, \dots, k \},
\]
where $\text{Mod}(X)$ denotes the mapping class group of $X$ and $\ell_{\alpha_i}(X)$ denotes the hyperbolic length of the unique geodesic representative of $\alpha_i$ on $X$. In this paper we prove the following result, originally conjectured by Wolpert in the case of simple closed multi-curves. \\

\begin{theorem}
	\label{theo:wolp_conj}
	If $\gamma$ is either simple or filling, the counting function $c(X,\gamma,\mathbf{b},L)$ is asymptotic as $L \to \infty$ to a polynomial in $L$ of degree $6g-6+2n$, i.e., the following limit exists and is a positive real number,
	\[
	c(X,\gamma,\mathbf{b}) := \lim_{L \to \infty} \frac{c(X,\gamma,\mathbf{b},L)}{L^{6g-6+2n}}.
	\]
\end{theorem}
$ $

The limit $c(X,\gamma,\mathbf{b}) \in \mathbf{R}_{>0}$ in Theorem \ref{theo:wolp_conj} can be described in terms of the geometry of $X$, the topology of $\gamma$, and the vector $\mathbf{b}$; see Theorem \ref{theo:simple_comp_count} for a description in the simple case and Theorem \ref{theo:filling_comp_count} for a description in the filling case. These seemingly unrelated descriptions are unified in Theorem \ref{theo:top_int_asymp_spec}. The simple case of Theorem \ref{theo:wolp_conj} can be strengthened to obtain asymptotics of countings that also keep track of the class of the multi-geodesics in the space of projective measured geodesic laminations; see Theorem \ref{theo:simple_proj_comp_count}. \\

Mirzakhani's results and techniques in \cite{Mir08b} can be used to establish asymptotics for countings of simple closed multi-curves with respect to length functions much more general than total hyperbolic length. Recent generalizations by several authors also deal with countings of objects much more general than simple closed multi-curves; see \cite{ES16}, \cite{EPS16}, \cite{RS19}. The simple case of Theorem \ref{theo:wolp_conj} does not fit into this framework; see Remark \ref{rem:not_consequence}. Instead, our proof of the simple case of Theorem \ref{theo:wolp_conj} draws inspiration from ideas introduced by Margulis in his thesis, see \cite{Mar04} for an English translation. Using general averaging and \textit{unfolding} techniques for parametrized countings, we reduce the proof of the simple case of Theorem \ref{theo:wolp_conj} to an application of equidistribution results for analogues of \textit{expanding horoballs} on moduli spaces of hyperbolic surfaces. A first version of these results was established by Mirzakhani in \cite{Mir07b} and was later generalized by the author in \cite{Ara19b}. To prove the filling case of Theorem \ref{theo:wolp_conj} we use techniques introduced by Mirzakhani in \cite{Mir16}. \\

A result closely related to the simple case of Theorem \ref{theo:wolp_conj} was independently established by Liu in \cite{Liu19}. In forthcoming work of Erlandsson and Souto, see \cite{ES20}, other results related to the simple case of Theorem \ref{theo:wolp_conj} are discussed.\\

\textit{Setting.}  Let $g,n \geq 0$ be a pair of non-negative integers satisfying $2 - 2g - n < 0$.  For the rest of this paper we fix a connected, oriented, smooth surface $S_{g,n}$ of genus $g$ with $n$ punctures (and negative Euler characteristic). \\

\textit{Notation.} Denote by $\mathcal{T}_{g,n}$ the Teichmüller space of marked, oriented, complete, finite area hyperbolic structures on $S_{g,n}$ up to isotopy, by $\text{Mod}_{g,n}$ the mapping class group of $S_{g,n}$, and by $\mathcal{M}_{g,n} := \mathcal{T}_{g,n}/\text{Mod}_{g,n}$ the moduli space of oriented, complete, finite area hyperbolic structures on $S_{g,n}$. Denote by $\mathcal{ML}_{g,n}$ the space of measured geodesic laminations on $S_{g,n}$ and by $P\mathcal{ML}_{g,n}:= \mathcal{ML}_{g,n}/\mathbf{R}_{>0}$ the space of projective measured geodesic laminations on $S_{g,n}$. The projective class of a measured geodesic laminations $\lambda \in \mathcal{ML}_{g,n}$ will be denoted by $\overline{\lambda} \in P\mathcal{ML}_{g,n}$. \\

Let $\alpha := (\alpha_1,\dots,\alpha_k)$ with $k \geq 1$ be an ordered tuple of closed curves on $S_{g,n}$, \textit{ordered closed multi-curve} for short, and $X \in \mathcal{T}_{g,n}$. The \textit{hyperbolic length vector} of $\alpha$ with respect to $X$ is given by
\[
\vec{\ell}_{\alpha}(X) := (\ell_{\alpha_1}(X),\dots,\ell_{\alpha_k}(X)) \in (\mathbf{R}_{>0})^k,
\]
where, for every $i \in \{1,\dots,k\}$, $\ell_{\alpha_i}(X) >0 $ denotes the hyperbolic length of the unique geodesic representative of $\alpha_i$ on $X$. Given a vector $\mathbf{a} := (a_1,\dots,a_k) \in (\mathbf{Q}_{>0})^k$ of positive rational weights on the components of $\alpha$, consider the \textit{rational closed multi-curve} on $S_{g,n}$ given by 
\begin{equation}
\label{eq:weighted_curve}
\mathbf{a} \cdot \alpha := a_1 \alpha_1 + \dots + a_k \alpha_k.
\end{equation}
The \textit{total hyperbolic length} of $\mathbf{a} \cdot \alpha$ with respect to $X$ is given by
\[
\ell_{\mathbf{a} \cdot \alpha}(X) := \mathbf{a} \cdot \vec{\ell}_\alpha(X) = a_1 \ell_{\alpha_1}(X) + \cdots + a_k \ell_{\alpha_k}(X) > 0.
\]
Unless otherwise specified, the term \textit{length} will always refer to \textit{hyperbolic length}.\\

\textit{Mirzakhani's counting results.} Let $\gamma := (\gamma_1,\dots,\gamma_k)$ with $k \geq 1$ be an ordered closed multi-curve on $S_{g,n}$, $\mathbf{a} := (a_1,\dots,a_k) \in (\mathbf{Q}_{>0})^k$ be a vector of positive rational weights on the components of $\gamma$, and $X \in \mathcal{T}_{g,n}$. For every $L > 0$ consider the counting function 
\[
t(X,\gamma,\mathbf{a},L) := \# \{\alpha \in \text{Mod}_{g,n} \cdot \gamma \ | \ \ell_{\mathbf{a} \cdot \alpha}(X) \leq L \}.
\]
This function does not depend on the marking of $X \in \mathcal{T}_{g,n}$ but only on the corresponding hyperbolic structure $X \in \mathcal{M}_{g,n}$. Mirzakhani's counting results describe the asymptotic behavior of $t(X,\gamma,\mathbf{a},L)$ as $L \to \infty$ when $\gamma$ is either \textit{simple}, i.e., the components of $\gamma$ are simple, pairwise disjoint, and pairwise non-isotopic, or \textit{filling}, i.e., the components of $\gamma$ cut $S_{g,n}$ into polygons with at most one puncture in their interior. \\

To give a precise statement of Mirzakhani's counting results, we first introduce some terminology. Consider the subgroup
\[
\text{Stab}(\gamma) = \bigcap_{i=1}^k \text{Stab}(\gamma_i) \subseteq \text{Mod}_{g,n}
\]
of all mapping classes of $S_{g,n}$ that fix every component of $\gamma$ up to isotopy. Let $\mu_{\text{wp}}$ be the Weil-Petersson measure on $\mathcal{T}_{g,n}$ and $\widetilde{\mu}_{\text{wp}}^\gamma$ be the local pushforward of $\mu_{\text{wp}}$ to $\mathcal{T}_{g,n}/\text{Stab}(\gamma)$. In \cite[Corollary 5.2]{Mir08b} and \cite[Theorem 8.1]{Mir16}, Mirzakhani showed that if $\gamma$ is either simple or filling, the following limit exists and is a positive rational number,
\[
r(\gamma,\mathbf{a}) := \lim_{L \to \infty} \frac{\widetilde{\mu}_{\text{wp}}^\gamma(\{Y \in \mathcal{T}_{g,n}/\text{Stab}(\gamma) \ | \ \ell_{\mathbf{a} \cdot \gamma}(Y) \leq L\})}{L^{6g-6+2n}}.
\]
$ $

In the case where $\gamma$ is simple, Mirzakhani gave an explicit formula for computing $r(\gamma,\mathbf{a})$, see \cite[Proposition 5.1]{Mir08b}. More concretely, letting $\mathbf{x} := (x_1,\dots,x_k)$ be the standard coordinates of $(\mathbf{R}_{\geq0})^k$ and $d\mathbf{x} := dx_1 \cdots dx_k$ be the standard measure of $(\mathbf{R}_{\geq0})^k$, there exists an explicit polynomial $W_{g,n}(\gamma,\mathbf{x})$ of degree $6g-6+2n-k$ on the $\mathbf{x}$ variables, all of whose non-zero monomials are of top degree, with non-negative rational coefficients, and which has $x_1 \cdots x_k$ as a factor, such that the following result holds. \\

\begin{proposition}[Mirzakhani]
	\label{prop:mir_freq}
	If $\gamma$ is simple,
	\[
	r(\gamma,\mathbf{a}) = \int_{\mathbf{a} \cdot \mathbf{x} \leq 1} W_{g,n}(\gamma,\mathbf{x}) \cdot d\mathbf{x}.
	\]
\end{proposition}
$ $

\begin{remark}
	Up to a constant, $W_{g,n}(\gamma,\mathbf{x})$ is equal to $x_1 \cdots x_k$ times the sum of the top degree monomials of the product of the Weil-Petersson volume polynomials of the moduli spaces of bordered Riemann surfaces associated to the components of the surface obtained by cutting $S_{g,n}$ along $\gamma$. See \S 2 for a precise definition.
\end{remark}
$ $

Let $\mu_{\text{Thu}}$ be the Thurston measure on $\mathcal{ML}_{g,n}$. Consider the function $B \colon \mathcal{M}_{g,n} \to \mathbf{R}_{>0}$ which to every $X \in \mathcal{M}_{g,n}$ assigns the value
\begin{equation}
\label{eq:mir_fn}
B(X):= \mu_{\text{Thu}}(\{\lambda \in \mathcal{ML}_{g,n} \ | \ \ell_\lambda(X) \leq 1\}),
\end{equation}
where $\ell_{\lambda}(X) > 0$ denotes the hyperbolic length of $\lambda$ with respect to $X$. We refer to this function as the \textit{Mirzakhani function}. Roughly speaking, $B(X)$ measures the shortness of simple closed geodesics on $X$. Let $\widehat{\mu}_\text{wp}$ be the local pushforward of the Weil-Petersson measure $\mu_{\text{wp}}$ on $\mathcal{T}_{g,n}$ to  $\mathcal{M}_{g,n} := \mathcal{T}_{g,n}/ \text{Mod}_{g,n}$. By work of Mirzakhani, see \cite[Proposition 3.2, Theorem 3.3]{Mir08b}, $B$ is continuous, proper, and integrable with respect to $\widehat{\mu}_{\text{wp}}$. Define
\begin{equation}
\label{eq:b_gn}
b_{g,n} := \int_{\mathcal{M}_{g,n}} B(X) \ \widehat{\mu}_{\text{wp}}(X) < +\infty.
\end{equation}
$ $

\begin{remark}
	In \cite{AA19}, upper and lower bounds of the same order describing the behavior of $B$ near the cusp of $\mathcal{M}_{g,n}$ are established. In particular, it is proved that $B$ is square-integrable with respect to $\widehat{\mu}_{\text{wp}}$.
\end{remark}
$ $

The following theorem, corresponding to \cite[Theorem 6.1]{Mir08b} and \cite[Theorem 1.1]{Mir16}, shows that if $\gamma$ is either simple or filling, the counting function $t(X,\gamma,\mathbf{a},L)$ is asymptotic to $L^{6g-6+2n}$ times an explicit constant;  see the footnote to \cite[Theorem 5.1]{Wri19} as well as \cite{RS20} and \cite{ES20} for discussions concerning the case of general closed multi-curves.\\

\begin{theorem}[Mirzakhani]
	\label{theo:mir_count}
	If $\gamma$ is either simple or filling,
	\[
	\lim_{L \to \infty} \frac{t(X,\gamma,\mathbf{a},L)}{L^{6g-6+2n}} = \frac{B(X) \cdot r(\gamma,\mathbf{a})}{b_{g,n}}.
	\]
\end{theorem}
$ $

\begin{remark}
	In \cite{EMM19}, Eskin, Mirzakhani, and Mohammadi improved Theorem \ref{theo:mir_count} in the case where $\gamma$ is simple by obtaining a power saving error term for the asymptotics of $t(X,\gamma,\mathbf{a},L)$. Their methods are very different from the ones in \cite{Mir08b} and \cite{Mir16}, and rely on the exponential mixing rate of the Teichmüller geodesic flow.
\end{remark}
$ $

As a follow-up question to Theorem \ref{theo:mir_count}, it is natural to ask whether the hyperbolic length vectors (and not just the total hyperbolic lengths) with respect to complete, finite area hyperbolic structures of multi-geodesics in mapping class group orbits of simple or filling closed multi-curves equidistribute near infinity. A first result in this direction can be found in \cite[Theorem 1.2]{Mir16}. The main goal of this paper is to answer this question in as much generality as possible.\\

\textit{Length spectra of ordered closed multi-curves.} Let $\gamma := (\gamma_1,\dots,\gamma_k)$ with $k \geq 1$ be an ordered closed multi-curve on $S_{g,n}$ and $X \in \mathcal{T}_{g,n}$. One can record, with multiplicity, the hyperbolic length vector with respect to $X$ of every ordered closed multi-curve in the mapping class group orbit of $\gamma$ by considering the counting measure on $(\mathbf{R}_{\geq 0})^k$ given by
\[
\mu_{\gamma,X} := \sum_{\alpha \in \text{Mod}_{g,n} \cdot \gamma} \delta_{\vec{\ell}_{\alpha}(X)}.
\]
This measure does not depend on the marking of $X \in \mathcal{T}_{g,n}$ but only on the corresponding hyperbolic structure $X \in \mathcal{M}_{g,n}$. We refer to this measure as the \textit{length spectrum} of $\gamma$ with respect to $X$.\\

To study the asymptotic behavior of $\mu_{\gamma,X}$, consider the rescaled counting measures $\{\mu_{\gamma,X}^L\}_{L>0}$ on $(\mathbf{R}_{\geq 0})^k$ given by 
\[
\mu_{\gamma,X}^L := \sum_{\alpha \in \text{Mod}_{g,n} \cdot \gamma} \delta_{\frac{1}{L} \cdot \vec{\ell}_\alpha(X)}.
\]
$ $

\textit{Asymptotics of length spectra of ordered simple closed multi-curves.} One of the main results of this paper is the following theorem, which describes the behavior near infinity of the length spectra of ordered simple closed multi-curves with respect to complete, finite area hyperbolic structures. \\

\begin{theorem}
	\label{theo:length_spec_simple}
	Let $\gamma:=(\gamma_1,\dots,\gamma_k)$ with $1 \leq k \leq 3g-3+n$ be an ordered simple closed multi-curve on $S_{g,n}$ and $X \in \mathcal{M}_{g,n}$. Then,
	\[
	\lim_{L \to \infty} \frac{\mu_{\gamma,X}^L}{L^{6g-6+2n}} = \frac{B(X)}{b_{g,n}} \cdot W_{g,n}(\gamma,\mathbf{x}) \cdot d\mathbf{x}
	\]
	in the weak-$\star$ topology for measures on $(\mathbf{R}_{\geq 0})^k$.
\end{theorem}
$ $

Theorem \ref{theo:length_spec_simple} and Portmanteau's theorem directly yield the following strong version, originally conjectured by Wolpert, of the simple case of Theorem \ref{theo:wolp_conj}; a closely related result was also recently established by Liu in  \cite{Liu19}.\\

\begin{theorem}
	\label{theo:simple_comp_count}
	Let $X \in \mathcal{M}_{g,n}$, $\gamma:=(\gamma_1,\dots,\gamma_k)$ with $1 \leq k \leq 3g-3+n$ be an ordered simple closed multi-curve on $S_{g,n}$, and $\mathbf{b} := (b_1,\dots,b_k) \in (\mathbf{R}_{>0})^k$. Then,
	\[
	\lim_{L \to \infty} \frac{c(X,\gamma,\mathbf{b},L)}{L^{6g-6+2n}} = \frac{B(X)}{b_{g,n}} \cdot \int_{\prod_{i=1}^k [0,b_i]} W_{g,n}(\gamma,\mathbf{x}) \cdot d\mathbf{x}.
	\]
\end{theorem}
$ $

\begin{remark}
	The simple case of Theorem \ref{theo:mir_count} can be deduced directly from Theorem \ref{theo:simple_comp_count}. This provides an alternative proof of such result which is independent of Mirzakhani's original work in \cite{Mir08b}.
\end{remark}
$ $

\begin{remark}
	\label{rem:not_consequence}
	Theorem \ref{theo:length_spec_simple} is not a direct consequence of the simple clase of Theorem \ref{theo:mir_count}. Indeed, simplices of $(\mathbf{R}_{\geq0})^k$ of the form
	\[
	\Delta_\mathbf{a} := \{(x_1,\dots,x_k) \in (\mathbf{R}_{\geq0})^k \ | \ a_1x_1 + \cdots + a_k x_k \leq 1 \}
	\]
	with $\mathbf{a} := (a_1,\dots,a_k) \in (\mathbf{R}_{>0})^k$ arbitrary do not generate the $\sigma$-algebra of Borel measurable subsets of $(\mathbf{R}_{\geq0})^k$. Moreover, Mirzakhani's counting results for mapping class group orbits of rational multi-curves, see \cite[Theorem 6.4]{Mir08b}, cannot be used to deduce Theorem \ref{theo:length_spec_simple} directly as the notion of \textit{length of topological components} does not extend continuously from the dense subset of rational multi-curves to all $\mathcal{ML}_{g,n}$.
\end{remark}
$ $

\begin{remark}
	Letting $\mathbf{b} := (1,\dots,1) \in (\mathbf{R}_{> 0})^k$ in Theorem \ref{theo:simple_comp_count} gives asymptotics for the counting functions
	\[
	m(X,\gamma,L) := \#\left\lbrace\alpha := (\alpha_1,\dots,\alpha_k) \in \text{Mod}_{g,n} \cdot \gamma \ \bigg\vert \ \max_{i=1,\dots,k} \ell_{\alpha_i}(X) \leq L \right\rbrace.
	\]
	$ $
\end{remark}

\textit{Main ideas of the proof of Theorem \ref{theo:length_spec_simple}.} Let $\gamma:=(\gamma_1,\dots,\gamma_k)$ with $1 \leq k \leq 3g-3+n$ be an ordered simple closed multi-curve on $S_{g,n}$ and $X \in \mathcal{M}_{g,n}$. It is convenient to rephrase Theorem \ref{theo:length_spec_simple} in the following equivalent way. Let $f \colon (\mathbf{R}_{\geq0})^k \to \mathbf{R}_{\geq 0}$ be an arbitrary non-negative, continuous, compactly supported function. For every $L > 0$ consider the counting function
\[
c(X,\gamma,f,L) := \int_{\mathbf{R}^k} f(\mathbf{x}) \  d\mu_{\gamma,X}^L(\mathbf{x}) = \sum_{\alpha \in \text{Mod}_{g,n} \cdot \gamma} f\left( \textstyle\frac{1}{L} \cdot \vec{\ell}_\alpha(X)\right).
\]
Theorem \ref{theo:length_spec_simple} is equivalent to the following result.\\

\begin{theorem}
	\label{theo:length_spec_simple_red}
	Let $f \colon (\mathbf{R}_{\geq0})^k \to \mathbf{R}_{\geq 0}$ be a non-negative, continuous, compactly supported function. Then,
	\[
	\lim_{L \to \infty} \frac{c(X,\gamma,f,L)}{L^{6g-6+2n}} = \frac{B(X)}{b_{g,n}} \cdot \int_{\mathbf{R}^k} f(\mathbf{x}) \cdot W_{g,n}(\gamma, \mathbf{x}) \cdot d\mathbf{x}.
	\]
\end{theorem}
$ $

Our proof of Theorem \ref{theo:length_spec_simple_red} is inspired by ideas introduced by Margulis in his thesis, see \cite{Mar04} for an English translation. The upshot of the proof is the following: approaching countings directly for a particular hyperbolic structure $X$ is rather hard but averaging them over nearby points in $\mathcal{M}_{g,n}$ should make them more tractable. After suitably \textit{spreading out} and averaging the countings over nearby points, \textit{unfolding} such averages on an appropriate intermediate cover reduces the proof of Theorem \ref{theo:length_spec_simple_red} to the question of whether certain analogues of \textit{expanding horoballs} on $\mathcal{M}_{g,n}$ equidistribute. Such equidistribution results were established by the author in \cite{Ara19b} building on ideas introduced by Mirzakhani in \cite{Mir07b}. \\

\begin{remark}
	\label{rem:effective_intro_1}
	If the analogues of expanding horoballs on $\mathcal{M}_{g,n}$ alluded to in the previous paragraph equidistributed at a polynomial rate, see Remark \ref{rem:effective_def} for a precise statement of this condition, the methods in our proof would yield an effective version of Theorem \ref{theo:length_spec_simple_red} with a power saving error term.
\end{remark}
$ $

\textit{Length and projective class spectra of ordered simple closed multi-curves.} Let $\gamma := (\gamma_1,\dots,\gamma_k)$ with $1 \leq k \leq 3g-3+n$ be an ordered simple closed multi-curve on $S_{g,n}$, $X \in \mathcal{T}_{g,n}$, and $\mathbf{a}:=(a_1,\dots,a_k) \in (\mathbf{Q}_{>0})^k$ be a vector of positive rational weights on the components of $\gamma$. One can record, with multiplicity, the hyperbolic length vector with respect to $X$ and the projective class in $P\mathcal{ML}_{g,n}$ with respect to the weights $\mathbf{a}$ of every ordered closed multi-curve in the mapping class group orbit of $\gamma$ by considering the counting measure on $(\mathbf{R}_{\geq 0})^k \times P\mathcal{ML}_{g,n}$ given by
\[
\nu_{\gamma,X,\mathbf{a}} := \sum_{\alpha \in \text{Mod}_{g,n} \cdot \gamma} \delta_{\vec{\ell}_\alpha(X)} \otimes \delta_{\overline{\mathbf{a} \cdot \alpha}}.
\]
This measure depends on marking of $X \in \mathcal{T}_{g,n}$. We refer to this measure as the \textit{length and projective class spectrum} of $\gamma$ with respect to $X$ and $\mathbf{a}$.\\

To study the asymptotic behavior of $\nu_{\gamma,X,\mathbf{a}}$, consider the family of rescaled counting measures $\{\nu_{\gamma,X,\mathbf{a}}^L\}_{L>0}$ on $(\mathbf{R}_{\geq 0})^k \times P\mathcal{ML}_{g,n}$ given by 
\[
\nu_{\gamma,X,\mathbf{a}}^L := \sum_{\alpha \in \text{Mod}_{g,n} \cdot \gamma} \delta_{\frac{1}{L} \cdot \vec{\ell}_\alpha(X)} \otimes \delta_{\overline{\mathbf{a} \cdot \alpha}}.
\]
$ $

\textit{Asymptotics of length and projective class spectra of ordered simple closed multi-curves.} Given $X \in \mathcal{T}_{g,n}$, let $\mu_{\text{Thu}}^X$ be the measure on $P\mathcal{ML}_{g,n}$ which to every Borel measurable subset $V \subseteq P\mathcal{ML}_{g,n}$ assigns the value
\[
\mu_{\text{Thu}}^X(V) := \mu_{\text{Thu}}(\{\lambda \in \mathcal{ML}_{g,n} \ | \ \ell_\lambda(X) \leq 1,  \ \overline{\lambda}\in V\}).
\]
A more refined application of the ideas in the proof of Theorem \ref{theo:length_spec_simple} yields the following stronger result, which describes the behavior near infinity of the length and projective class spectra of ordered simple closed multi-curves with respect to complete, finite area hyperbolic structures and positive rational weights.\\

\begin{theorem}
	\label{theo:length_proj_spec_simple}
	Let $\gamma := (\gamma_1,\dots,\gamma_k)$ with $1 \leq k \leq 3g-3+n$ be an ordered simple closed multi-curve on $S_{g,n}$, $X \in \mathcal{T}_{g,n}$, and $\mathbf{a}:=(a_1,\dots,a_k) \in (\mathbf{Q}_{>0})^k$ be a vector of positive rational weights on the components of $\gamma$. Then,
	\[
	\lim_{L \to \infty} \frac{\nu_{\gamma,X,\mathbf{a}}^L}{L^{6g-6+2n}} = \frac{1}{b_{g,n}} \cdot W_{g,n}(\gamma,\mathbf{x}) \cdot d\mathbf{x} \otimes \mu_{\text{Thu}}^X 
	\]
	in the weak-$\star$ topology for measures on $(\mathbf{R}_{\geq 0})^k \times P\mathcal{ML}_{g,n}$.
\end{theorem}
$ $

\begin{remark}
	Theorem \ref{theo:length_spec_simple} can be deduced from Theorem \ref{theo:length_proj_spec_simple} by taking pushforwards under the map $(\mathbf{R}_{\geq0})^k \times P\mathcal{ML}_{g,n} \to (\mathbf{R}_{\geq0})^k$ which projects to the first coordinate.
\end{remark}
$ $

Let $X \in \mathcal{T}_{g,n}$, $\gamma := (\gamma_1,\dots,\gamma_k)$ with $1 \leq k \leq 3g-3+n$ be an ordered simple closed multi-curve on $S_{g,n}$, $\mathbf{b}:=(b_1,\dots,b_k) \in (\mathbf{R}_{>0})^k$, $\mathbf{a}:=(a_1,\dots,a_k) \in (\mathbf{Q}_{>0})^k$, and $V \subseteq P\mathcal{ML}_{g,n}$ be a continuity subset of the Thurston measure class, i.e., $V$ is a Borel measurable subset satisfying
\[
\mu_{\text{Thu}}(\{\lambda \in \mathcal{ML}_{g,n} \ | \ \overline{\lambda} \in \partial V \}) = 0.
\]
For every $L > 0$ consider the counting function
\begin{align*}
&c(X,\gamma,\mathbf{b},L,\mathbf{a},V)\\
&:= \#\left\lbrace
\begin{array}{c | l}
\alpha := (\alpha_i)_{i=1}^k \in \text{Mod}_{g,n} \cdot \gamma
& \ \ell_{\alpha_i}(X) \leq b_i L, \ \forall i=1,\dots,k,\\
& \ \overline{\mathbf{a} \cdot \gamma} \in V.\\
\end{array} \right\rbrace.
\end{align*}
This counting function depends on marking of $X \in \mathcal{T}_{g,n}$. The following strengthening of Theorem \ref{theo:simple_comp_count} is a direct consequence of Theorem \ref{theo:length_spec_simple}, Lemma \ref{lem:thu_meas_zero}, and Portmanteau's theorem.\\

\begin{theorem}
	\label{theo:simple_proj_comp_count}
	Let $X \in \mathcal{T}_{g,n}$, $\gamma := (\gamma_1,\dots,\gamma_k)$ with $1 \leq k \leq 3g-3+n$ be an ordered simple closed multi-curve on $S_{g,n}$, $\mathbf{b}:=(b_1,\dots,b_k) \in (\mathbf{R}_{>0})^k$, $\mathbf{a}:=(a_1,\dots,a_k) \in (\mathbf{Q}_{>0})^k$, and $V \subseteq P\mathcal{ML}_{g,n}$ be a continuity subset of the Thurston measure class. Then,
	\[
	\lim_{L \to \infty} \frac{c(X,\gamma,\mathbf{b},L,\mathbf{a},V)}{L^{6g-6+2n}} = \frac{\mu_{\text{Thu}}^X(V)}{b_{g,n}} \cdot \int_{\prod_{i=1}^k [0,b_i]} W_{g,n}(\gamma,\mathbf{x}) \cdot d\mathbf{x}.
	\]
\end{theorem}
$ $

\begin{remark}
	\label{rem:effective_intro_2}
	Just as in the case of Theorem \ref{theo:length_spec_simple_red}, if certain analogues of expanding horoballs on the bundle of unit length measured geodesics laminations over $\mathcal{M}_{g,n}$ equidistributed at a polynomial rate, the methods in our proof would yield an effective version of Theorem \ref{theo:length_spec_simple_red} with a power saving error term; see Remark \ref{rem:effective_def_2}.
\end{remark}
$ $

\textit{Topological factor of asymptotic length (and projective class) spectra of ordered simple closed multi-curves.} Let $\gamma := (\gamma_1,\dots,\gamma_k)$ with $k \geq 1$ be an ordered closed multi-curve on $S_{g,n}$. According to Theorem \ref{theo:length_spec_simple} (and Theorem \ref{theo:length_proj_spec_simple}), if $\gamma$ is simple, the asymptotic length (and projective class) spectrum of $\gamma$ with respect to any $X \in \mathcal{T}_{g,n}$ (and any vector $\mathbf{a}:=(a_1,\dots,a_k)\in (\mathbf{Q}_{>0})^k$ of positive rational weights on the components of $\gamma$) has a factor
\[
W_{g,n}(\gamma,\mathbf{x}) \cdot d\mathbf{x}
\]
which depends only on $\gamma$ and not on $X$ (or $\mathbf{a}$). We provide a purely topological description of this factor.\\

A measured geodesic laminations $\lambda \in \mathcal{ML}_{g,n}$ fills $S_{g,n}$ together with $\gamma$ if the geodesic representatives of the components of $\gamma$ and the topological support of $\lambda$ cut $S_{g,n}$ into polygons with no ideal vertices. Let $\mathcal{ML}_{g,n}(\gamma) \subseteq \mathcal{ML}_{g,n}$ be the open, dense, full measure subset of all measured geodesic laminations that together with $\gamma$ fill $S_{g,n}$. The stabilizer $\text{Stab}(\gamma) \subseteq \text{Mod}_{g,n}$ acts properly discontinuously on $\mathcal{ML}_{g,n}(\gamma)$, see Proposition \ref{prop:pd_Y}. Consider the measure $\mu_{\text{Thu}}^\gamma := \mu_{\text{Thu}}|_{\mathcal{ML}_{g,n}(\gamma)}$ on $\mathcal{ML}_{g,n}(\gamma)$ and let $\widetilde{\mu}_{\text{Thu}}^\gamma$ be its local pushforward to the quotient $\mathcal{ML}_{g,n}(\gamma)/\text{Stab}(\gamma)$. Consider  the map 
\[
I_\gamma \colon \mathcal{ML}_{g,n}(\gamma) \to (\mathbf{R}_{\geq 0})^k
\]
which to every $\lambda \in \mathcal{ML}_{g,n}(\gamma) $ assigns the vector
\[
I_\gamma(\lambda) := (i(\gamma_1,\lambda), \dots, i(\gamma_k,\lambda)) \in  (\mathbf{R}_{\geq 0})^k
\]
and let
\[
\widetilde{I}_\gamma \colon \mathcal{ML}_{g,n}(\gamma)/\text{Stab}(\gamma) \to (\mathbf{R}_{\geq 0})^k
\]
be its induced map on the quotient $\mathcal{ML}_{g,n}(\gamma)/\text{Stab}(\gamma)$.\\
$ $

\begin{theorem}
	\label{theo:top_int_asymp_spec}
	If $\gamma$ is simple,
	\[
	W_{g,n}(\gamma,\mathbf{x}) \cdot d\mathbf{x} = (\widetilde{I}_\gamma)_*(\widetilde{\mu}_{\text{Thu}}^\gamma).
	\]
\end{theorem}
$ $

Our proof of Theorem \ref{theo:top_int_asymp_spec} uses Thurston's shear coordinates, see \cite[\S 9]{Thu98}, and the measure preserving properties of such coordinates established by Papadopoulos and Penner in the case of punctured surfaces,  see \cite[Corollary 4.2]{PP93}, and by Bonahon and S\"{o}zen in the case of closed surfaces, see \cite[Theorem 1]{BS01}. The characterization of the subset $\mathcal{ML}_{g,n}(\gamma) \subseteq \mathcal{ML}_{g,n}$ provided by Proposition	\ref{prop:fil_char} will also play a crucial in the proof as it will help us to deal with issues of non-compactness.\\

\textit{Asymptotics of length spectra of ordered filling closed multi-curves.} Using techniques introduced by Mirzakhani in \cite{Mir16}, we prove the following theorem, which describes the behavior near infinity of the length spectra of ordered filling closed multi-curves with respect to complete, finite area hyperbolic structures.\\

\begin{theorem}
	\label{theo:length_spec_filling}
	Let $\gamma:=(\gamma_1,\dots,\gamma_k)$ with $k \geq 1$ be an ordered filling closed multi-curve on $S_{g,n}$ and $X \in \mathcal{M}_{g,n}$. Then,
	\[
	\lim_{L \to \infty} \frac{\mu_{\gamma,X}^L}{L^{6g-6+2n}} = \frac{B(X)}{b_{g,n}} \cdot (\widetilde{I}_\gamma)_*\left(\widetilde{\mu}_{\text{Thu}}^\gamma\right)
	\]
	in the weak-$\star$ topology for measures on $(\mathbf{R}_{\geq 0})^k$.
\end{theorem}
$ $

Theorem \ref{theo:length_spec_filling}, Lemma \ref{lem:thu_meas_zero}, and Portmanteau's theorem directly yield the following strong version of the filling case of Theorem \ref{theo:wolp_conj}.\\

\begin{theorem}
	\label{theo:filling_comp_count}
	Let $X \in \mathcal{M}_{g,n}$, $\gamma:=(\gamma_1,\dots,\gamma_k)$ with $k\geq 1$ be an ordered filling closed multi-curve on $S_{g,n}$, and $\mathbf{b} := (b_1,\dots,b_k) \in (\mathbf{R}_{>0})^k$. Then,
	\[
	\lim_{L \to \infty} \frac{c(X,\gamma,\mathbf{b},L)}{L^{6g-6+2n}} = \frac{B(X)}{b_{g,n}} \cdot \widetilde{\mu}_{\text{Thu}}^\gamma\left(\{\lambda \in \mathcal{ML}_{g,n}(\gamma)/\text{Stab}(\gamma) \ | \ i(\lambda,\gamma_i) \leq b_i\} \right).
	\]
\end{theorem}
$ $

\begin{remark}
	As highlighted by Mirzakhani in \cite{Mir16}, applying the methods in the proof of Theorem \ref{theo:filling_comp_count} to get an effective version of the same theorem with a power saving error term seems rather hard.
\end{remark}
$ $

\textit{Organization of the paper.} In \S 2 we present the background material necessary to understand the proofs of the main results. In \S 3 we present the proof of Theorem \ref{theo:length_spec_simple_red} and discuss how refining the ideas
in this proof leads to a proof of Theorem \ref{theo:length_proj_spec_simple}. In \S4 we prove Theorem \ref{theo:top_int_asymp_spec}. In \S 5 we briefly review the techniques introduced by Mirzakhani in \cite{Mir16} and use them to prove Theorem \ref{theo:length_spec_filling}.\\

\textit{Acknowledgments.} The author is very grateful to Alex Wright and Steven Kerckhoff for their invaluable advice, patience, and encouragement.\\

\section{Background material}

$ $

\begin{sloppypar}
\textit{The Thurston measure.} Train track coordinates induce a $(6g-6+2n)$-dimensional piecewise integral linear (PIL) structure on the space $\mathcal{ML}_{g,n}$ of measured geodesic laminations on $S_{g,n}$; see \cite[\S 3.1]{PH92} for details. By work of Masur, see \cite[Theorem 2]{Mas85}, there exists a unique (up to scaling) non-zero, locally finite, $\text{Mod}_{g,n}$-invariant, Lebesgue class measure on $\mathcal{ML}_{g,n}$. Several definitions of such a measure can be found in the literature. We will consider the measure coming from the symplectic structure of $\mathcal{ML}_{g,n}$.\\
\end{sloppypar}

More precisely, consider the $\text{Mod}_{g,n}$-invariant symplectic form $\omega_{\text{Thu}}$ on the PIL manifold $\mathcal{ML}_{g,n}$ induced by train track coordinates; see \cite[\S 3.2]{PH92} for an explicit definition. The top exterior power $v_\text{Thu} := \frac{1}{(3g-3+n)!} \bigwedge_{i=1}^{3g-3+n} \omega_{\text{Thu}}$ induces a non-zero, locally finite, $\text{Mod}_{g,n}$-invariant, Lebesgue class measure $\mu_{\text{Thu}}$ on $\mathcal{ML}_{g,n}$. We refer to this measure as the \textit{Thurston measure} of $\mathcal{ML}_{g,n}$.\\

This measure satisfies the following scaling property:
\begin{equation}
\label{eq:thu_meas_scale}
\mu_{\text{Thu}}(t \cdot A) = t^{6g-6+2n} \cdot \mu_{\text{Thu}}(A)
\end{equation}
for every Borel measurable subset $A \subseteq \mathcal{ML}_{g,n}$ and every $t > 0$. In particular, the following lemma applies; see \cite[Page 24]{EU18} for a proof.\\

\begin{lemma}
	\label{lem:thu_meas_zero}
	Let $\Omega$ be a topological space endowed with a continuous $(\mathbf{R}_{>0})$-action and $\mu$ be a measure on $\Omega$ such that the following scaling property holds for some $k > 0$:
	\[
	\mu(t \cdot A) = t^{k} \cdot \mu(A)
	\]
	for every Borel measurable subset $A \subseteq \Omega$ and every $t > 0$. Let $f \colon \Omega \to \mathbf{R}_{\geq 0}$ be a non-negative, homogeneous, continuous function. Then, for every $c > 0$,
	\[
	\mu(f^{-1}(\{c\})) = 0.
	\]
\end{lemma}
$ $

\textit{Dehn-Thurston coordinates.} Let $\mathcal{P}:=(\gamma_1,\dots,\gamma_{3g-3+n})$ be a pair of pants decomposition of $S_{g,n}$. The following theorem, originally due to Dehn in the case of integral multi-curves and later extended by Thurston to the case of general measured geodesic laminations, gives an explicit parametrization of $\mathcal{ML}_{g,n}$ in terms of intersection numbers $m_i \in \mathbf{R}_{\geq 0}$ and twisting numbers $t_i \in \mathbf{R}$ with respect to the components of $\mathcal{P}$; see \S 1.2 in \cite{PH92} and \S 8.3.9 in \cite{Mar16} for details.\\

\begin{theorem}
	\label{theo:dehn_thurston_coordinates}
	The intersection and twisting numbers $(m_i,t_i)_{i=1}^{3g-3+n}$ with respect to the components of $\mathcal{P}$ give a parametrization of $\mathcal{ML}_{g,n}$ by the set
	\[
	\Theta := \left\lbrace(m_i,t_i) \in (\mathbf{R}_{\geq 0} \times \mathbf{R})^{3g-3+n} \ | \ m_i = 0 \Rightarrow t_i \geq 0, \ \forall i=1,\dots, 3g-3+n \right\rbrace.
	\]
\end{theorem}
$ $

We refer to any parametrization as in Theorem \ref{theo:dehn_thurston_coordinates} as a set of \textit{Dehn-Thurston coordinates} of $\mathcal{ML}_{g,n}$ adapted to $\mathcal{P}$ and to the set $\Theta$ as the \textit{parameter space} of such parametrization. The Thurston measure $\mu_{\text{Thu}}$ on $\mathcal{ML}_{g,n}$ is precisely the Lebesgue measure on $\Theta$.\\

\textit{The Mirzakhani measure.} Over $\mathcal{T}_{g,n}$ consider the bundle $P^1 \mathcal{T}_{g,n}$ of unit length measured geodesic laminations. More precisely,
\[
P^1 \mathcal{T}_{g,n} := \{(X,\lambda) \in \mathcal{T}_{g,n} \times \mathcal{ML}_{g,n} \ | \ \ell_{\lambda}(X) = 1\}.
\]
For every marked hyperbolic structure $X \in \mathcal{T}_{g,n}$, consider the measure $\mu_{\text{Thu}}^X$ on the fiber  $P^1_X\mathcal{T}_{g,n}$ of the bundle $P^1\mathcal{T}_{g,n}$ above $X$, which to every Borel measurable subset $A \subseteq P^1_X\mathcal{T}_{g,n}$ assigns the value
\[
\mu_{\text{Thu}}^X(A) := \mu_{\text{Thu}}([0,1] \cdot A).
\]
On the bundle $P^1\mathcal{T}_{g,n}$ one obtains a measure $\nu_{\text{Mir}}$, called the \textit{Mirzakhani measure} of $P^1\mathcal{T}_{g,n}$, by considering the disintegration formula
\[
d\nu_{\text{Mir}}(X,\lambda) := d\mu_{\text{Thu}}^X(\lambda) \  d\mu_{\text{wp}}(X).
\]
$ $

The mapping class group $\text{Mod}_{g,n}$ acts diagonally on $P^1 \mathcal{T}_{g,n}$ in a properly discontinuous way preserving the Mirzakhani measure $\nu_{\text{Mir}}$. The quotient $P^1 \mathcal{M}_{g,n} := P^1\mathcal{T}_{g,n}/\text{Mod}_{g,n}$ is the bundle of unit length measured geodesic laminations over the moduli space $\mathcal{M}_{g,n}$. Locally pushing forward the measure $\nu_{\text{Mir}}$ on $P^1\mathcal{T}_{g,n}$ under the quotient map $P^1\mathcal{T}_{g,n} \to P^1\mathcal{M}_{g,n}$ yields a measure $\widehat{\nu}_{\text{Mir}}$ on $P^1\mathcal{M}_{g,n}$, called the \textit{Mirzakhani measure} of $P^1\mathcal{M}_{g,n}$. The pushforward of $\widehat{\nu}_{\text{Mir}}$ under the bundle map $\pi \colon P^1\mathcal{M}_{g,n} \to \mathcal{M}_{g,n}$ is given by
\[
d\pi_*(\widehat{\nu}_{\text{Mir}})(X) = B(X) \ d\widehat{\mu}_\text{wp}(X),
\]
where $B \colon \mathcal{M}_{g,n} \to \mathbf{R}_{>0}$ is the Mirzakhani function defined in (\ref{eq:mir_fn}).
The total mass of $P^1 \mathcal{M}_{g,n}$ with respect to $\widehat{\nu}_{\text{Mir}}$ is precisely given by
\[
\widehat{\nu}_{\text{Mir}}(P^1 \mathcal{M}_{g,n}) = \int_{\mathcal{M}_{g,n}} B(X) \ d\widehat{\mu}_{\text{wp}}(X) = b_{g,n}.
\]
In particular, it is finite by (\ref{eq:b_gn}).\\

\textit{Horoball segment measures.} Let $\gamma := (\gamma_1,\dots,\gamma_k)$ with $1 \leq k \leq 3g-3+n$ be an ordered simple closed multi-curve on $S_{g,n}$ and $f \colon (\mathbf{R}_{\geq 0})^k \to \mathbf{R}_{\geq0}$ be a bounded, compactly supported, Borel measurable function with non-negative values and which is not almost everywhere zero with respect to the Lebesgue measure class. For every $L > 0$ consider the \textit{horoball segment} $B_\gamma^{f,L} \subseteq \mathcal{T}_{g,n}$ given by 
\[
B_\gamma^{f,L} := \{X \in \mathcal{T}_{g,n} \ | \ \vec{\ell}_{\gamma}(X) \in L \cdot \text{supp}(f) \}.
\]
Every such horoball segment supports a \textit{horoball segment measure} $\mu_\gamma^{f,L}$ defined as 
\begin{equation}
\label{eq:horoball_meas_def}
d \mu_\gamma^{f,L}(X) := f\left(\textstyle \frac{1}{L} \cdot \vec{\ell}_{\gamma}(X) \right) \ d\mu_{\text{wp}}(X).
\end{equation}
This measure is $\text{Stab}(\gamma)$-invariant. To get a locally finite horoball segment measure on $\mathcal{M}_{g,n}$ we need to get rid of the redundancies that arise when taking pushforwards. For this purpose we consider the intermediate cover
\[
\mathcal{T}_{g,n} \to \mathcal{T}_{g,n}/\text{Stab}(\gamma) \to \mathcal{M}_{g,n}.
\]
Let $\widetilde{\mu}_\gamma^{f,L}$ be the local pushforward of $\mu_\gamma^{f,L}$ to $\mathcal{T}_{g,n}/\text{Stab}(\gamma)$ and $\widehat{\mu}_{\gamma}^{f,L}$ be the pushforward of $\widetilde{\mu}_\gamma^{f,L}$ to $\mathcal{M}_{g,n}$. \\

Let $\mathbf{a}:=(a_1,\dots,a_k) \in (\mathbf{Q}_{>0})^k$ be a vector of positive rational weights on the components of $\gamma$. As $\gamma$ is simple, $\mathbf{a} \cdot \gamma$ as defined in (\ref{eq:weighted_curve}) belongs to $\mathcal{ML}_{g,n}(\mathbf{Q})$. The horoball segment measures $\mu_\gamma^{f,L}$ on $\mathcal{T}_{g,n}$ also give rise to \textit{horoball segment measures} $\nu_{\gamma,\mathbf{a}}^{f,L}$ on the bundle $P^1 \mathcal{T}_{g,n}$ by considering the disintegration formula
\[
d \nu_{\gamma,\mathbf{a}}^{f,L}(X,\lambda) := d \delta_{\mathbf{a} \cdot \gamma/ \ell_{\mathbf{a} \cdot \gamma}(X)}(\lambda) \ d\mu_\gamma^{f,L}(X),
\]
where $\delta$ denotes point masses. This measure is $\text{Stab}(\gamma)$-invariant as well.  In analogy with the case above, to get locally finite horoball segment measures on $P^1\mathcal{M}_{g,n}$ we consider the intermediate cover
\[
P^1\mathcal{T}_{g,n} \to P^1\mathcal{T}_{g,n}/\text{Stab}(\gamma) \to P^1\mathcal{M}_{g,n}.
\]
Let $\widetilde{\nu}_{\gamma,\mathbf{a}}^{f,L}$ be the local pushforward of $\nu_{\gamma,\mathbf{a}}^{f,L}$ to $P^1\mathcal{T}_{g,n}/\text{Stab}(\gamma)$ and  $\widehat{\nu}_{\gamma,\mathbf{a}}^{f,L}$ the pushforward of $\widetilde{\nu}_{\gamma,\mathbf{a}}^{f,L}$ to $P^1\mathcal{M}_{g,n}$.\\

One can check, see Proposition \ref{prop:total_hor_meas} below, that the measures $\widehat{\mu}_\gamma^{f,L}$ and $\widehat{\nu}_{\gamma,\mathbf{a}}^{f,L}$ are finite. We denote by $m_\gamma^{f,L}$ the total mass of the measures $\widehat{\mu}_\gamma^{f,L}$ and $\widehat{\nu}_{\gamma,\mathbf{a}}^{f,L}$, i.e.,
\[
m_\gamma^{f,L} := \widehat{\mu}_\gamma^{f,L}(\mathcal{M}_{g,n}) = \widehat{\nu}_{\gamma,\mathbf{a}}^{f,L}(P^1\mathcal{M}_{g,n}) <+\infty.
\]
$ $

The main tool used in the proof of Theorem \ref{theo:length_proj_spec_simple} is the following result, which shows that horoball segment measures on $P^1\mathcal{M}_{g,n}$ equidistribute with respect to $\widehat{\nu}_{\text{Mir}}$. This result is an analogue of the classical equidistribution theorem for expanding horoballs on homogeneous spaces, see for instance \cite{KM96}. This result is proved in \cite{Ara19b}, expanding on ideas introduced by Mirzakhani in \cite{Mir07b}.\\

\begin{theorem}
	\label{theo:horoball_equid}
	In the weak-$\star$ topology for measures on $P^1\mathcal{M}_{g,n}$,
	\[
	\lim_{L \to \infty} \frac{\widehat{\nu}_{\gamma,\mathbf{a}}^{f,L}}{m_\gamma^{f,L}} = \frac{\widehat{\nu}_{\text{Mir}}}{b_{g,n}}.
	\]
\end{theorem}
$ $

Taking pushforwards under the bundle map $\pi \colon P^1\mathcal{M}_{g,n} \to \mathcal{M}_{g,n}$ in the statement of Theorem \ref{theo:horoball_equid}, we deduce the following corollary, which shows that horoball segment measures on $\mathcal{M}_{g,n}$ equidistribute with respect to $B(X) \cdot d\widehat{\mu}_{\text{wp}}(X)$. This corollary is the main tool used in the proof of Theorem \ref{theo:length_spec_simple}.\\

\begin{corollary}
	\label{cor:horoball_equid}
	In the weak-$\star$ topology for measures on $\mathcal{M}_{g,n}$,
	\[
	\lim_{L \to \infty} \frac{\widehat{\mu}_\gamma^{f,L}}{m_\gamma^{f,L}} = \frac{B(X) \cdot d\widehat{\mu}_{\text{wp}}(X)}{b_{g,n}}.
	\]
\end{corollary}
$ $

\textit{Teichmüller and moduli spaces of hyperbolic surfaces with geodesic boundary.} Let $g',n',b' \in \mathbf{Z}_{\geq 0}$ be a triple of non-negative integers satisfying $2 - 2g' - n' - b' < 0$. Consider a fixed connected, oriented surface $S_{g',n'}^{b'}$ of genus $g'$ with $n'$ punctures and $b'$ labeled boundary components $\beta_1,\dots,\beta_{b'}$. Let $\mathbf{L}:= (L_i)_{i=1}^{b'} \in (\mathbf{R}_{>0})^{b'}$ be a vector of positive real numbers. \\

We denote by $\mathcal{T}_{g',n'}^{b'}(\mathbf{L})$ the \textit{Teichmüller space} of  marked, oriented, complete, finite area hyperbolic structures on $S_{g',n'}^{b'}$ with labeled geodesic boundary components whose lengths are given by $\mathbf{L}$. The \textit{mapping class group} of $S_{g',n'}^{b'}$, denoted $\text{Mod}_{g',n'}^{b'}$, is the group of isotopy classes of orientation preserving diffeomorphisms of $S_{g',n'}^{b'}$ that set-wise fix each boundary component. The quotient $\mathcal{M}_{g',n'}^{b'}(\mathbf{L}) := \mathcal{T}_{g',n'}^{b'}(\mathbf{L}) / \text{Mod}_{g',n'}^{b'}$ is the \textit{moduli space} of oriented, complete, finite area hyperbolic structures on $S_{g',n'}^{b'}$ with labeled geodesic boundary components whose lengths are given by $\mathbf{L}$. \\

Consider the \textit{total Weil-Petersson volume} of the moduli space $\mathcal{M}_{g',n'}^{b'}(\mathbf{L})$, 
\[
V_{g',n'}^{b'}(\mathbf{L}) := \text{Vol}_{\text{wp}}(\mathcal{M}_{g',n'}^{b'}(\mathbf{L})).
\]
The following remarkable theorem due to Mirzakhani, see \cite[Theorem 6.1]{Mir07a} and \cite[Theorem 1.1]{Mir07c}, shows that $V_{g',n'}^{b'}(\mathbf{L})$ behaves like a polynomial on the $\mathbf{L}$ variables.\\

\begin{theorem}
	\label{theo:vol_pol}
	The total Weil-Petersson volume
	\[
	V_{g',n'}^{b'}(L_1,\dots,L_{b'})
	\]
	is a polynomial of degree $3g'-3+n'+b'$ on the variables $L_1^2,\dots,L_{b'}^2$. Moreover, if we denote
	\[
	V_{g',n'}^{b'}(L_1,\dots,L_{b'}) = \sum_{\substack{\alpha \in (\mathbf{Z}_{\geq 0})^{b'}, \\ |\alpha| \leq 3g'-3+n'+b'} } c_\alpha \cdot L_1^{2\alpha_1} \cdots L_{b'}^{2\alpha_{b'}},
	\]
	where $|\alpha| := \alpha_1 + \cdots + \alpha_{b'}$ for every $\alpha \in (\mathbf{Z}_{\geq 0})^{b'}$, then $c_\alpha \in \mathbf{Q}_{>0} \cdot \pi^{6g'-6+2n' +2b' - 2|\alpha|}$. In particular, the leading coefficients of $V_{g',n'}^{b'}(L_1,\dots,L_{b'})$ belong to $\mathbf{Q}_{> 0}$.
\end{theorem}
$ $

\begin{remark}
	If the surface $S_{g',n'}^{b'}$ is a pair of pants, i.e., if $g' = 0$ and $n'+b' = 3$, then, for any $\mathbf{L} := (L_i)_{i=1}^{b'} \in (\mathbf{R}_{> 0})^{b'}$, the moduli space $\mathcal{M}_{g',n'}^{b'}(\mathbf{L})$ has exactly one point. We will adopt the convention 
	\[
	V_{g',n'}^{b'}(\mathbf{L}) := 1.
	\]
\end{remark}
$ $

\textit{The polynomials $W_{g,n}(\gamma,\mathbf{x})$.} Given a simple closed curve $\alpha$ on $S_{g,n}$, let
\[
\text{Stab}_0(\alpha) \subseteq \text{Mod}_{g,n}
\]
be the subgroup of all mapping classes of $S_{g,n}$ that fix $\alpha$ (up to isotopy) together with its orientations (although $\alpha$ is unoriented, it admits two possible orientations which are being required to be fixed). More generally, given an ordered simple closed multi-curve $\gamma := (\gamma_1,\dots,\gamma_k)$ on $S_{g,n}$ with $1 \leq k \leq 3g-3+n$, let
\[
\text{Stab}_0(\gamma) := \bigcap_{i=1}^k \text{Stab}_0(\gamma_i) \subseteq \text{Mod}_{g,n}
\]
be the subgroup of all mapping classes of $S_{g,n}$ that fix each component of $\gamma$ (up to isotopy) together with their respective orientations.\\

For the rest of this discussion fix an ordered simple closed multi-curve $\gamma := (\gamma_1,\dots,\gamma_k)$ on $S_{g,n}$ with $1 \leq k \leq 3g-3+n$.  Let $S_{g,n}(\gamma)$ be the (potentially disconnected) oriented topological surface with boundary obtained by cutting $S_{g,n}$ along the components of $\gamma$. Let $c \in \mathbf{Z}_{>0}$ be the number of components of $S_{g,n}(\gamma)$ and $\{\Sigma_j\}_{j=1}^c$ be the components of $S_{g,n}(\gamma)$. For every$j \in \{1,\dots,c\}$ let $g_j,n_j,b_j \in \mathbf{Z}_{\geq 0}$ be the triple of non-negative integers satisfying $2 - 2g_j - n_j - b_j < 0$ such that $\Sigma_j$ is homeomorphic to $S_{g_j,n_j}^{b_j}$. Given a vector $\mathbf{x} := (x_i)_{i=1}^k \in (\mathbf{R}_{>0})^k$, for every $j \in \{1,\dots,c\}$ let $\mathbf{x}_j \in (\mathbf{R}_{>0})^{b_j}$ be the subvector of $\mathbf{x}$ whose entries correspond to the boundary components of $\Sigma_j$.\\

Let $\rho_{g,n}(\gamma)$ be the number of components of $\gamma$ that bound (on any of its sides) a component of $S_{g,n}(\gamma)$ which is a torus with one boundary component. Let $\sigma_{g,n}(\gamma) \in \mathbf{Q}_{>0}$ be the rational number
\[
\sigma_{g,n}(\gamma) := \frac{\prod_{j=1}^c |K_{g_j,n_j}^{b_j}|}{|\text{Stab}_0(\gamma)\cap K_{g,n}|},
\]
where $K_{g_j,n_j}^{b_j} \triangleleft \text{Mod}_{g_j,n_j}^{b_j}$ is the kernel of the mapping class group action on $\mathcal{T}_{g_j,n_j}^{b_j}$ and $K_{g,n} \triangleleft \text{Mod}_{g,n}$ is the kernel of the mapping class group action on $\mathcal{T}_{g,n}$. For example, if $g =2$, $n=0$, and $\gamma$ is a separating simple closed curve on $S_{2,0}$, then $\sigma_{2,0}(\gamma) = 4/2 = 2$.\\

For vectors $\mathbf{x} := (x_i)_{i=1}^k \in (\mathbf{R}_{>0})^k$ consider the polynomial
\[
V_{g,n}(\gamma,\mathbf{x}) := \frac{1}{[\text{Stab}(\gamma):\text{Stab}_0(\gamma)]} \cdot \sigma_{g,n}(\gamma) \cdot 2^{-\rho_{g,n}{(\gamma)}} \cdot \prod_{j=1}^c V_{g_j,n_j}^{b_j}(\mathbf{x}_j) \cdot x_1 \cdots x_k.
\]
By Theorem \ref{theo:vol_pol}, $V_{g,n}(\gamma,\mathbf{x})$ is a polynomial of degree $6g-6+2n-k$, with non-negative coefficients, and rational leading coefficients. Denote by 
\begin{equation}
\label{eq:W_gn}
W_{g,n}(\gamma,\mathbf{x}) := V_{g,n}^{\text{top}}(\gamma,\mathbf{x}) 
\end{equation}
the polynomial obtained by adding up all the leading (maximal degree) monomials of $V_{g,n}(\gamma,\mathbf{x})$. The polynomial $W_{g,n}(\gamma,\mathbf{x})$ only depends on $g$, $n$, and the $\text{Mod}_{g,n}$-orbit of $\gamma$. \\
\begin{example}
	Table \ref{table:genus_2} contains the polynomials $W_{2,0}(\gamma,x_1,\dots,x_k)$ for all possible $\text{Mod}_{2,0}$-orbits of ordered simple closed multi-curves $\gamma:=(\gamma_1,\dots,\gamma_k)$ on $S_{2,0}$. These polynomial were computed using $(\ref{eq:W_gn})$ and the tables in \cite[\S B]{Do13}. 
	
	\begin{table}
		\centering
		
		\begin{tabular}{ | c | c | }
			\hline
			$\gamma := (\gamma_1,\dots,\gamma_k)$ & $W_{2,0}(\gamma,x_1,\dots,x_k)$ \\ \hline 
			\ & \\[2pt]
			
			\quad \includegraphics[width=0.3\textwidth]{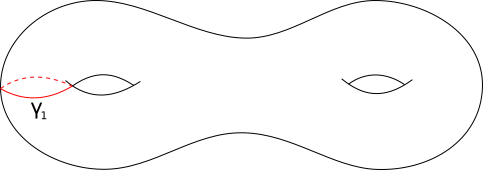} \quad \ & \begin{tabular}{c} $\frac{1}{96} x_1^5$ \\ \\ \\ \end{tabular}\\[9pt] 
			\hline
			\ & \\[2pt]
			
			\quad \includegraphics[width=0.3\textwidth]{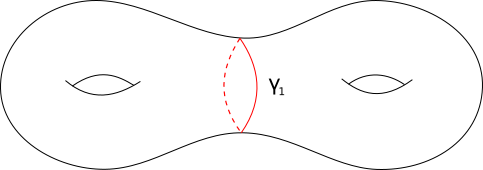}  \quad \ & \begin{tabular}{c} $\frac{1}{4608} x_1^5$ \\ \\ \\ \end{tabular} \\[9pt] 
			\hline
			\ &  \\[2pt]
			
			\quad \includegraphics[width=0.3\textwidth]{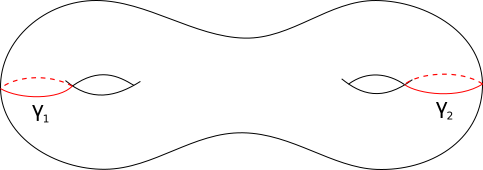}  \quad \ & \begin{tabular}{c} $\frac{1}{4} x_1^3x_2 + \frac{1}{4} x_1x_2^3$ \\ \\ \\ \end{tabular} \\[9pt] 
			\hline
			\ &  \\[2pt]
			
			\quad \includegraphics[width=0.3\textwidth]{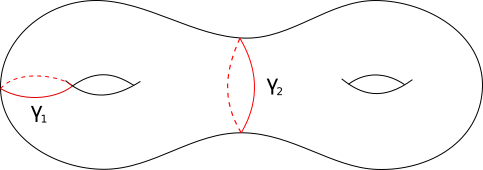}  \quad \ & \begin{tabular}{c} $\frac{1}{96} x_1 x_2^3$ \\ \\ \\ \end{tabular} \\[9pt] 
			\hline
			\ &  \\[2pt]
			
			\quad \includegraphics[width=0.3\textwidth]{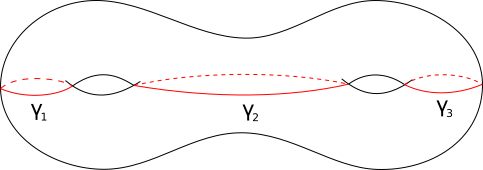}  \quad \ & \begin{tabular}{c} $\frac{1}{2} x_1 x_2 x_3$ \\ \\ \\ \end{tabular}\\[9pt] 
			\hline
			\ &  \\[2pt]
			
			\quad \includegraphics[width=0.3\textwidth]{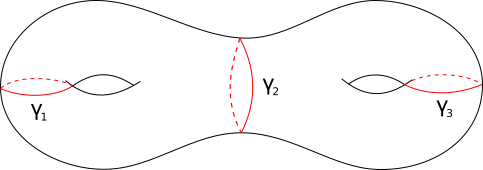}  \quad \ & \begin{tabular}{c} $\frac{1}{4} x_1 x_2 x_3$ \\ \\ \\ \end{tabular}\\[9pt] 
			\hline
		\end{tabular}
		
		\vspace{+0.7cm}
		
		\caption{Polynomials $W_{2,0}(\gamma,x_1,\dots,x_k)$ for all possible $\text{Mod}_{2,0}$-orbits of ordered simple closed multi-curves $\gamma := (\gamma_1,\dots,\gamma_k)$ on $S_{2,0}$.}
		\label{table:genus_2}
	\end{table}
	
\end{example}
$ $

\begin{example}
	For every pair of pants decomposition $\mathcal{P}:= (\gamma_1,\dots,\gamma_{3g-3+n})$ of $S_{g,n}$  there exists $k \in \mathbf{Z}_{\geq 0}$ such that
	\[
	W_{g,n}(\mathcal{P},x_1,\dots,x_{3g-3+n}) = 2^{-k} \cdot x_1 \cdots x_{3g-3+n}.
	\]
\end{example}
$ $

\textit{Total mass of horoball segment measures.} Let $\gamma := (\gamma_1,\dots,\gamma_k)$ with $1 \leq k \leq 3g-3+n$ be an ordered simple closed multi-curve on $S_{g,n}$ and $f \colon (\mathbf{R}_{\geq 0})^k \to \mathbf{R}_{\geq0}$ be a bounded, compactly supported, Borel measurable function with non-negative values and which is not almost everywhere zero with respect to the Lebesgue measure class. As mentioned above, the horoball segment measures $\widehat{\mu}_\gamma^{f,L}$ on $\mathcal{M}_{g,n}$ and $\widehat{\nu}_{\gamma,\mathbf{a}}^{f,L}$ on $P^1\mathcal{M}_{g,n}$ are finite. One can actually compute explicit formulas for their total mass $m_\gamma^{f,L}$ in terms of the polynomial $V_{g,n}(\gamma,\mathbf{x})$ and use them describe the asymptotics of $m_\gamma^{f,L}$ as $L \to \infty$ in terms of the polynomial $W_{g,n}(\gamma,\mathbf{x})$. See \cite[Proposition 3.1]{Ara19b} for a proof of the following result.\\

\begin{proposition}
	\label{prop:total_hor_meas}
	For every $L > 0$,
	\[
	m_\gamma^{f,L} = \int_{\mathbf{R}^k} f(\mathbf{x}) \cdot V_{g,n}(\gamma,L \cdot \mathbf{x}) \cdot L^k \  d \mathbf{x},
	\]
	where $d\mathbf{x} := dx_1 \cdots dx_k$. In particular,
	\[
	\lim_{L \to \infty} \frac{m_\gamma^{f,L}}{L^{6g-6+2n}} = \int_{\mathbf{R}^k} f(\mathbf{x}) \cdot W_{g,n}(\gamma,\mathbf{x}) \ d \mathbf{x}.
	\]
\end{proposition}
$ $

\textit{The symmetric Thurston metric.} Consider the \textit{asymmetric Thurston metric} $d_\text{Thu}'$ on $\mathcal{T}_{g,n}$ which to every pair $X,Y \in \mathcal{T}_{g,n}$ assigns the distance
\[
d_\text{Thu}'(X,Y) := \sup_{\lambda \in \mathcal{ML}_{g,n}} \log \left(\frac{\ell_\lambda(Y)}{\ell_\lambda(X)} \right).
\]
As this metric is asymmetric, it is convenient to consider the \textit{symmetric Thurston metric} $d_\text{Thu}$ on $\mathcal{T}_{g,n}$ which to every pair $X,Y \in \mathcal{T}_{g,n}$ assigns the distance
\[
d_\text{Thu}(X,Y) := \max\{d_\text{Thu}'(X,Y), d_\text{Thu}'(Y,X)\}.
\]
A pair $X,Y \in \mathcal{T}_{g,n}$ satisfies $d_\text{Thu}(X,Y) \leq \epsilon$ for some $\epsilon > 0$ precisely when
\begin{equation}
\label{eq:thurston_met}
e^{-\epsilon} \ell_\lambda(X) \leq \ell_\lambda(Y) \leq e^\epsilon  \ell_{\lambda}(X), \ \forall \lambda \in \mathcal{ML}_{g,n}.
\end{equation}
The metric $d_\text{Thu}$ induces the usual topology on $\mathcal{T}_{g,n}$. We denote by $U_X(\epsilon) \subseteq \mathcal{T}_{g,n}$ the closed ball of radius $\epsilon > 0$ centered at $X \in \mathcal{T}_{g,n}$ with respect to $d_\text{Thu}$. For more details on the theory of the asymmetric and symmetric Thurston metrics, see \cite{Thu98} and \cite{Pap15}. \\

\textit{The Yamabe space.} Let $\mathcal{Y}_{g,n}$ be the \textit{Yamabe space} of all complete, finite area, constant negative curvature metrics on $S_{g,n}$ up to isotopy. One can identify
\[
\mathcal{Y}_{g,n} = \left(\mathbf{R}_{>0}\right) \times \mathcal{T}_{g,n},
\]
where $(t,X) \in \left(\mathbf{R}_{>0}\right) \times \mathcal{T}_{g,n}$ corresponds to the scaling $t\cdot X \in \mathcal{Y}_{g,n}$ of the hyperbolic metric $X \in \mathcal{T}_{g,n}$ which scales lengths by $t > 0$. Let $\overline{\mathcal{Y}_{g,n}}$ be the \textit{enlarged Yamabe space} obtained by adjoining a copy of $\mathcal{ML}_{g,n}$ to $\overline{\mathcal{Y}_{g,n}}$,
\[
\overline{\mathcal{Y}_{g,n}} := \mathcal{Y}_{g,n} \sqcup \mathcal{ML}_{g,n}.
\]
$ $

Consider the pairing $i \colon \overline{\mathcal{Y}_{g,n}}  \times \mathcal{ML}_{g,n} \to \mathbf{R}_{\geq 0}$ which to every $(\alpha,\mu) \in \overline{\mathcal{Y}_{g,n}}  \times \mathcal{ML}_{g,n}$ assigns the value
\[
i(\alpha,\mu) := \left\lbrace
\begin{array}{ccl}
t \cdot \ell_\mu(X) & \text{if} & \alpha:=(t,X) \in \mathcal{Y}_{g,n}, \\
i(\lambda,\mu) & \text{if} & \alpha:=\lambda \in \mathcal{ML}_{g,n}. \\ 
\end{array} \right. 
\]
This pairing is homogenous with respect to the natural $\mathbf{R}_{>0}$ actions on each coordinate. On $\overline{\mathcal{Y}_{g,n}}$ consider the weakest topology making this pairing continuous. With this topology $\mathcal{T}_{g,n} = \{1\} \times \mathcal{T}_{g,n} \subseteq \mathcal{Y}_{g,n}$ and $\mathcal{ML}_{g,n} \subseteq \mathcal{Y}_{g,n}$ are embedded. By work of Thurston, see for instance \cite[Theorem 8.7]{FLP12}, $\overline{\mathcal{Y}_{g,n}}$ is projectively compact, that is, $P\overline{\mathcal{Y}_{g,n}} := \overline{\mathcal{Y}_{g,n}}/\mathbf{R}_{>0}$ is compact. The natural $\text{Mod}_{g,n}$ action on $\overline{\mathcal{Y}_{g,n}}$ is continuous.\\

\textit{Properly discontinuous stabilizer actions.} Consider the subset 
\[
\overline{\mathcal{Y}_{g,n}}(\gamma) := \mathcal{Y}_{g,n} \sqcup \mathcal{ML}_{g,n}(\gamma) \subseteq \overline{\mathcal{Y}_{g,n}}.
\]
If $n=0$, the following result is a direct consequence of \cite[Proposition 4.1]{EM18}; the same arguments can be adapted to obtain a proof in the case $n > 0$.\\

\begin{proposition}
	\label{prop:pd_Y}
	The group $\text{Stab}(\gamma)$ acts properly discontinuously on  $\overline{\mathcal{Y}_{g,n}}(\gamma)$.
\end{proposition}
$ $

Proposition \ref{prop:pd_Y} implies in particular that $\text{Stab}(\gamma)$ acts properly discontinuously on $\mathcal{ML}_{g,n}(\gamma)$. It follows that, as was mentioned in \S 1, $\widetilde{\mu}_{\text{Thu}}^\gamma$, the local pushforward of the measure $\mu_{\text{Thu}}^\gamma := \mu_{\text{Thu}}|_{\mathcal{ML}_{g,n}(\gamma)}$ on $\mathcal{ML}_{g,n}(\gamma)$ to the quotient $\mathcal{ML}_{g,n}(\gamma)/ \allowbreak\text{Stab}(\gamma)$, is well defined. \\

\textit{Thurston's shear coordinates.} Let $\mu$ be a maximal geodesic lamination on $S_{g,n}$. It is not required for $\mu$ to support an invariant transverse measure. Consider the open, dense, full measure subset $\mathcal{ML}_{g,n}(\mu) \subseteq \mathcal{ML}_{g,n}$ of all measured geodesic laminations that together with $\mu$ fill $S_{g,n}$. More precisely, $\lambda \in \mathcal{ML}_{g,n}(\mu)$ if and only if $\mu$ and the topological support of $\lambda$ cut $S_{g,n}$ into polygons with no ideal vertices and with at most one puncture in their interior. Let $\text{Stab}(\mu) \subseteq \text{Mod}_{g,n}$ be the subgroup of all mapping classes of $S_{g,n}$ that stabilize $\mu$. In \cite{Thu98}, Thurston introduced a $\text{Stab}(\mu)$-equivariant global parametrization of $\mathcal{T}_{g,n}$,
\[
F_\mu \colon \mathcal{T}_{g,n} \to \mathcal{ML}_{g,n}(\mu),
\]
called the \textit{shear coordinates} of $\mathcal{T}_{g,n}$ with respect to $\mu$. Roughly speaking, this map sends $X \in \mathcal{T}_{g,n}$ to the \textit{transverse horocyclic foliation} $F_\mu(X)$ of $\mu$ on $X$. The $F_\mu(X)$-measure of a subarc of $\mu$ is given by the hyperbolic length of such arc on $X$. In particular, given any $X \in \mathcal{T}_{g,n}$ and any $\lambda \in \mathcal{ML}_{g,n}$,
\begin{equation}
\label{eq:shear_ineq}
i(F_\mu(X),\lambda) \leq \ell_{\lambda}(X).
\end{equation}
Moreover, if one of the components of $\mu$ is a simple closed curve $\gamma$ then 
\begin{equation}
\label{eq:shear_eq}
i(F_\mu(X),\gamma) = \ell_{\gamma}(X).
\end{equation}
$ $

By work of Papadopoulos and Penner, see \cite[Corollary 4.2]{PP93}, and of Bonahon and Sözen, see \cite[Theorem 1]{BS01}, if $n > 0$ and $\mu$ is an ideal geodesic triangulation of $S_{g,n}$, or if $n = 0$ and $\mu$ is a maximal geodesic lamination of $S_{g,n}$, the shear coordinates
\[
F_\mu \colon \mathcal{T}_{g,n} \to \mathcal{ML}_{g,n}(\mu)
\]
pull back the the restriction of Thurston symplectic form $\omega_{\text{Thu}}$ on $\mathcal{ML}_{g,n}(\mu)$ to the Weil-Petersson symplectic form $\omega_{\text{wp}}$ on $\mathcal{T}_{g,n}$. As a direct consequence of these results we obtain the following corollary.\\

\begin{corollary}
	\label{cor:shear_mp}
	Suppose that $n > 0$ and $\mu$ is a finite ideal geodesic triangulation of $S_{g,n}$, or that $n = 0$ and $\mu$ is a maximal geodesic lamination on $S_{g,n}$. Then the shear coordinates
	\[
	F_\mu \colon \mathcal{T}_{g,n} \to \mathcal{ML}_{g,n}(\mu)
	\]
	pull back the restriction of the Thurston measure $\mu_{\text{Thu}}$ on $\mathcal{ML}_{g,n}(\mu)$ to the Weil-Petersson measure $\mu_{\text{wp}}$ on $\mathcal{T}_{g,n}$.
\end{corollary}
$ $

By work of Papadopoulos, see \cite[Proposition 3.1]{Pap88} and \cite[Lemma 4.9]{Pap91}, the behavior of shear coordinates along sequence in $\mathcal{T}_{g,n}$ approaching the Thurston boundary $P\mathcal{ML}_{g,n}$ is well understood.\\

\begin{lemma}
	\label{lem:asymp_shear}
	Suppose that $n > 0$ and $\mu$ is a finite ideal geodesic triangulation of $S_{g,n}$, or that $n = 0$ and $\mu$ is a maximal geodesic lamination on $S_{g,n}$. Let $(X_n)_{n \in \mathbf{N}}$ be a sequence of points in $\mathcal{T}_{g,n}$ converging to a projective measured geodesic lamination on the Thurston boundary $P\mathcal{ML}_{g,n}$. Then, for every simple closed curve $\alpha$ on $S_{g,n}$ there exists a constant $C>0$ such that for every $n \in \mathbf{N}$,
	\[
	i(F_\mu(X_n),\alpha) \leq \ell_{\alpha}(X_n) \leq i(F_\mu(X_n),\alpha) + C.
	\]
\end{lemma}
$ $

\textit{Filling pairs of measured geodesic laminations.} A pair of measured geodesic laminations $\lambda,\mu \in \mathcal{ML}_{g,n}$ is said to \textit{fill} $S_{g,n}$ if the topological supports of $\lambda$ and $\mu$ cut $S_{g,n}$ into polygons with no ideal vertices and with at most one puncture in their interior. This condition can be characterized in terms of the intersection pairing of $\mathcal{ML}_{g,n}$ in the following way; see \cite[\S 1.2, \S4.3]{Mir08a} for more details.\\

\begin{proposition}
	\label{prop:ml_fil}
	A pair $\lambda,\mu \in \mathcal{ML}_{g,n}$ fills $S_{g,n}$ if and only if
	\[
	i(\lambda,\eta) + i(\mu,\eta) > 0, \ \forall \eta \in \mathcal{ML}_{g,n}.
	\]
\end{proposition}
$ $

\textit{Bers's Theorem.} The following version of Bers's theorem can be proved using arguments similar to those in the proof of \cite[Theorem 12.8]{FM11}. \\

\begin{theorem}
	\label{theo:bers}
	Let $1 \leq k \leq 3g-3+n$ and $\mathbf{b} := (b_1,\dots,b_k) \in (\mathbf{R}_{>0})^k$. There exists a constant $C \geq \max_{i=1,\dots,k}  b_i$ such that for any $X \in \mathcal{T}_{g,n}$ and any ordered simple closed multi-curve $\gamma := (\gamma_1,\dots,\gamma_k)$ on $S_{g,n}$ satisfying
	\[
	\ell_{\gamma_i}(X) \leq b_i, \ \forall i =1,\dots,k,
	\]
	there exists a completion $\mathcal{P} := (\gamma_1,\dots,\gamma_{3g-3+n})$ of $\gamma$ to a pair of pants decomposition of $S_{g,n}$ satisfying
	\[
	\ell_{\gamma_i}(X) \leq C,  \ \forall i =1,\dots,3g-3+n.
	\]
\end{theorem}
$ $

\section{Counting simple closed hyperbolic multi-geodesics}

$ $

\textit{Setting.} For the rest of this section, let $\gamma := (\gamma_1,\dots,\gamma_k)$ with $1 \leq k \leq 3g-3+n$ be an ordered simple closed multi-curve on $S_{g,n}$ and $X \in \mathcal{T}_{g,n}$  be a marked, oriented, complete, finite area hyperbolic structure on $S_{g,n}$.\\

\textit{Proof of Theorem \ref{theo:length_spec_simple_red}}. Let $f \colon (\mathbf{R}_{\geq0})^k \to \mathbf{R}_{\geq 0}$ be a non-negative, continuous, compactly supported function. As explained in \S 1, to prove Theorem \ref{theo:length_spec_simple_red} we proceed in two steps. First, considering $X$ as an element of $\mathcal{M}_{g,n}$, we \textit{spread out} and average the counting functions $c(X,\gamma,f,L)$ over points $Y\in \mathcal{M}_{g,n}$ near $X$. Second, we \textit{unfold} these averages over a suitable intermediate cover, reducing the proof of Theorem \ref{theo:length_spec_simple_red} to an application of Corollary \ref{cor:horoball_equid}.\\

\textit{Spreading out and averaging.} Given $\mathbf{x} := (x_i)_{i=1}^k \in (\mathbf{R}_{\geq0})^k$ and $\epsilon > 0$, let $N_\epsilon(\mathbf{x}) \subseteq (\mathbf{R}_{\geq0})^k$ be the subset 
\[
N_\epsilon(\mathbf{x}) := \{\mathbf{y} := (y_i)_{i=1}^k \in (\mathbf{R}_{\geq0})^k \ | \  e^{-\epsilon}  x_i \leq y_i \leq e^\epsilon x_i, \ \forall i =1,\dots, k\}.
\]
For every $\epsilon > 0$ consider the functions $f_\epsilon^{\max}, f_\epsilon^{\min}\colon  (\mathbf{R}_{\geq0})^k \to \mathbf{R}_{\geq0}$ which to every $\mathbf{x} \in  (\mathbf{R}_{\geq 0})^k$ assign the value
\[
f_\epsilon^{\max}(\mathbf{x}) := \max_{\mathbf{y} \in N_\epsilon(\mathbf{x})} f(\mathbf{y}), \quad f_\epsilon^{\min}(\mathbf{x}) := \min_{\mathbf{y} \in N_\epsilon(\mathbf{x})} f(\mathbf{y}).
\]
As $f \colon (\mathbf{R}_{\geq0})^k \to \mathbf{R}_{\geq 0}$ is continuous and compactly supported,
\[
\lim_{\epsilon \to 0} f_\epsilon^{\max}(\mathbf{x}) = f(\mathbf{x}), \quad \lim_{\epsilon \to 0} f_\epsilon^{\min}(\mathbf{x}) = f(\mathbf{x})
\]
uniformly over all $\mathbf{x} \in (\mathbf{R}_{\geq0})^k$.\\

Let $\epsilon > 0$ be arbitrary. Recall that $U_X(\epsilon) \subseteq \mathcal{T}_{g,n}$ denotes the closed ball of radius $\epsilon$ centered at $X$ with respect to the symmetric Thurston metric $d_\text{Thu}$. Let $\pi \colon \mathcal{T}_{g,n} \to \mathcal{M}_{g,n}$ be the quotient map. As highlighted in (\ref{eq:thurston_met}), $Y \in \mathcal{T}_{g,n}$ satisfies $d_{\text{Thu}}(X,Y) < \epsilon$ if and only if
\[
e^{-\epsilon} \ell_\lambda(X) \leq \ell_\lambda(Y) \leq e^\epsilon  \ell_{\lambda}(X), \ \forall \lambda \in \mathcal{ML}_{g,n}.
\]
In particular, for every $L > 0$, if $Y \in \mathcal{M}_{g,n}$ satisfies $Y \in \pi(U_X(\epsilon))$ then
\begin{equation}
\label{eq:count_comparison}
c(Y,\gamma,f_\epsilon^{\min},L) \leq c(X,\gamma,f,L) \leq c(Y,\gamma,f_\epsilon^{\max},L).
\end{equation}
$ $

Recall that $\widehat{\mu}_\text{wp}$ denotes the local pushforward of the Weil-Petersson measure $\mu_\text{wp}$ on $\mathcal{T}_{g,n}$ to the quotient $\mathcal{M}_{g,n}:= \mathcal{T}_{g,n}/\text{Mod}_{g,n}$. For every $\epsilon > 0$ let $\eta_\epsilon \colon \mathcal{M}_{g,n} \to \mathbf{R}_{\geq 0}$ be a continuous, compactly supported function satisfying
\begin{enumerate}
	\item $\text{supp}(\eta_\epsilon) \subseteq \pi(U_X(\epsilon))$,
	\item $ \displaystyle \int_{\mathcal{M}_{g,n}} \eta_\epsilon(Y) \ d\widehat{\mu}_{\text{wp}}(Y) = 1.$
\end{enumerate}
Multiplying (\ref{eq:count_comparison}) by $\eta_\epsilon(Y)$ and integrating over $\mathcal{M}_{g,n}$ with respect to $d\widehat{\mu}_{\text{wp}}(Y)$ one deduces
\begin{equation}
\label{eq:count_low}
\int_{\mathcal{M}_{g,n}} \eta_\epsilon(Y) \cdot c(Y,\gamma,f_\epsilon^{\min},L) \ d\widehat{\mu}_{\text{wp}}(Y) \leq c(X,\gamma,f,L),
\end{equation}
\begin{equation}
\label{eq:count_up}
c(X,\gamma,f,L) \leq \int_{\mathcal{M}_{g,n}} \eta_\epsilon(Y) \cdot c(Y,\gamma,f_\epsilon^{\max},L) \ d\widehat{\mu}_{\text{wp}}(Y).
\end{equation}
$ $
This concludes the spreading out and averaging step.\\

\textit{Unfolding averages.} Consider the  intermediate cover
\[
\mathcal{T}_{g,n} \to \mathcal{T}_{g,n} / \text{Stab}(\gamma) \to \mathcal{M}_{g,n}.
\]
Unfolding the integrals in (\ref{eq:count_low}) and (\ref{eq:count_up}) over $\mathcal{T}_{g,n} / \text{Stab}(\gamma)$ and pushing them back down to $\mathcal{M}_{g,n}$ in a suitable way will reduce the proof of Theorem \ref{theo:length_spec_simple_red} to an applicaton of Corollary \ref{cor:horoball_equid}. The following proposition describes this principle; see \S2 for the definition of the measures $\widehat{\mu}_{\gamma}^{h,L}$.\\

\begin{proposition}
	\label{prop:pull_push}
	Let $h \colon (\mathbf{R}_{\geq0})^k \to \mathbf{R}_{\geq 0}$ be a non-negative, continuous, compactly supported function. Then, for every $\epsilon > 0$ and every $L > 0$,
	\[
	\int_{\mathcal{M}_{g,n}} \eta_\epsilon(Y) \cdot c(Y,\gamma,h,L) \ d\widehat{\mu}_{\text{wp}}(Y) = \int_{\mathcal{M}_{g,n}} \eta_\epsilon(Y) \ d\widehat{\mu}_\gamma^{h,L}(Y).
	\]
\end{proposition}
$ $

\begin{proof}
	Let $\epsilon > 0$ and $L > 0$ be arbitrary. For every $Y \in \mathcal{M}_{g,n}$ one can rewrite the counting function $c(Y,\gamma,h,L)$ as follows:
	\begin{align*}
	c(Y,\gamma,h,L)  &= \sum_{\alpha \in \text{Mod}_{g,n} \cdot \gamma} h\left( \textstyle\frac{1}{L} \cdot \vec{\ell}_\alpha(Y)\right)\\
	&= \sum_{[\phi] \in \text{Mod}_{g,n}/\text{Stab}(\gamma)} h\left( \textstyle\frac{1}{L} \cdot \vec{\ell}_{\phi \cdot \gamma}(Y)\right) \\
	&= \sum_{[\phi] \in \text{Mod}_{g,n}/\text{Stab}(\gamma)} h\left( \textstyle\frac{1}{L} \cdot \vec{\ell}_\gamma(\phi^{-1} \cdot Y)\right) \\
	&= \sum_{[\phi] \in \text{Stab}(\gamma) \backslash \text{Mod}_{g,n}} h\left( \textstyle\frac{1}{L} \cdot \vec{\ell}_\gamma(\phi \cdot Y)\right). \\
	\end{align*}
	Let us record this as
	\begin{equation}
	\label{eq:count_massage}
	c(X,\gamma,h,L)  = \sum_{[\phi] \in \text{Stab}(\gamma) \backslash \text{Mod}_{g,n}} h\left( \textstyle\frac{1}{L} \cdot \vec{\ell}_\alpha(\phi \cdot X)\right).
	\end{equation}
	Let $p_\gamma \colon \mathcal{T}_{g,n}/\text{Stab}(\gamma) \to \mathcal{M}_{g,n}$ be the quotient map and $\widetilde{\eta}_\epsilon^\gamma \colon \mathcal{T}_{g,n}/\text{Stab}(\gamma) \to \mathbf{R}_{\geq 0}$ be the lift of $\eta_\epsilon$ given by $\widetilde{\eta}_\epsilon^\gamma := \eta_\epsilon \circ p_\gamma$. Recall that $\widetilde{\mu}_\text{wp}^\gamma$ denotes the local pushforward of the Weil-Petersson measure $\mu_\text{wp}$ on $\mathcal{T}_{g,n}$ to the quotient $\mathcal{T}_{g,n}/\text{Stab}(\gamma)$. It follows from (\ref{eq:count_massage}) that
	\[
	\int_{\mathcal{M}_{g,n}} \eta_\epsilon(Y) \cdot c(Y,\gamma,h,L) \ d \widehat{\mu}_\text{wp}(Y) = \int_{\mathcal{T}_{g,n}/\text{Stab}(\gamma)} \widetilde{\eta}_\epsilon^\gamma(Y) \cdot h\left(\textstyle\frac{1}{L} \cdot \vec{\ell}_\gamma(Y)\right) \
	d \widetilde{\mu}_{\text{wp}}^\gamma(Y).
	\]
	$ $
	
	By definition, see (\ref{eq:horoball_meas_def}), the measure $\mu_\gamma^{h,L}$ on $\mathcal{T}_{g,n}$ is given by
	\[
	d \mu_\gamma^{h,L}(Y) := h\left(\textstyle\frac{1}{L} \cdot \vec{\ell}_\gamma(Y)\right) \
	d \mu_{\text{wp}}(Y).
	\]
	Taking local pushforwards to $\mathcal{T}_{g,n}/\text{Stab}(\gamma)$ we deduce
	\[
	d \widetilde{\mu}_\gamma^{h,L}(Y) = h\left(\textstyle\frac{1}{L} \cdot \vec{\ell}_\gamma(Y)\right) \
	d \widetilde{\mu}_{\text{wp}}^\gamma(Y).
	\]
	It follows that 
	\[
	\int_{\mathcal{T}_{g,n}/\text{Stab}(\gamma)} \widetilde{\eta}_\epsilon^\gamma(Y) \cdot h\left(\textstyle\frac{1}{L} \cdot \vec{\ell}_\gamma(Y)\right) \
	d \widetilde{\mu}_{\text{wp}}^\gamma(Y) = \int_{\mathcal{T}_{g,n}/\text{Stab}(\gamma)} \widetilde{\eta}_\epsilon^\gamma(Y)
	\ d \widetilde{\mu}_\gamma^{h,L}(Y).
	\]
	As $\widehat{\mu}_\gamma^{h,L}$ is the pushforward of $\widetilde{\mu}_\gamma^{h,L}$ to $\mathcal{M}_{g,n}$,
	\[
	\int_{\mathcal{T}_{g,n}/\text{Stab}(\gamma)} \widetilde{\eta}_\epsilon^\gamma(Y)
	\ d \widetilde{\mu}_\gamma^{h,L}(Y,\alpha) = \int_{\mathcal{M}_{g,n}} \eta_\epsilon(Y)
	\ d \widehat{\mu}_\gamma^{h,L}(Y).
	\]
	$ $
	
	Putting everything together we deduce
	\[
	\int_{\mathcal{M}_{g,n}} \eta_\epsilon(Y) \cdot c(Y,\gamma,h,L) \ d\widehat{\mu}_{\text{wp}}(Y) = \int_{\mathcal{M}_{g,n}} \eta_\epsilon(Y) \ 	d\widehat{\mu}_\gamma^{h,L}(Y).
	\]
	This finishes the proof.
\end{proof}
$ $

\textit{Application of Corollary \ref{cor:horoball_equid}.} We are now ready to prove Theorem \ref{theo:length_spec_simple_red}. Corollary \ref{cor:horoball_equid} and Proposition \ref{prop:total_hor_meas} will play a fundamental role in the proof.\\

\begin{proof}[Proof of Theorem \ref{theo:length_spec_simple_red}]
	By Proposition \ref{prop:total_hor_meas}, given any non-negative, continuous, compactly supported function $h \colon (\mathbf{R}_{\geq 0})^k \to \mathbf{R}_{\geq 0}$, 
	\begin{equation}
	\label{eq:r(gamma,h)}
	r(\gamma,h) := \lim_{L \to \infty} \frac{m_\gamma^{h,L}}{L^{6g-6+2n}} = \int_{\mathbf{R}^k} h(\mathbf{x}) \cdot W_{g,n}(\gamma,\mathbf{x}) \cdot d \mathbf{x}.
	\end{equation}
	Proving Theorem \ref{theo:length_spec_simple_red} is then equivalent to showing that
	\begin{equation}
		\label{eq:count_lb}
		r(\gamma, f) \cdot \frac{B(X)}{b_{g,n}} \leq \liminf_{L \to \infty} \frac{c(X,\gamma,f,L)}{L^{6g-6+2n}},
	\end{equation}
	\begin{equation}
		\label{eq:count_ub}
		\limsup_{L \to \infty} \frac{c(X,\gamma,f,L)}{L^{6g-6+2n}} \leq r(\gamma, f) \cdot \frac{B(X)}{b_{g,n}}.
	\end{equation}
	$ $
	
	We first verify (\ref{eq:count_lb}). Let $\epsilon > 0$ and $L > 0$ be arbitrary. Consider $h := f_\epsilon^{\min}$. By (\ref{eq:count_low}) and Proposition \ref{prop:pull_push},
	\[
	\int_{\mathcal{M}_{g,n}} \eta_\epsilon(Y) \ d\widehat{\mu}_{\gamma}^{h,L}(Y) \leq c(X,\gamma,f,L).
	\]
	Dividing this inequality by $m_\gamma^{h,L} > 0$ we get
	\[
	\int_{\mathcal{M}_{g,n}} \eta_\epsilon(Y) \ \frac{d\widehat{\mu}_{\gamma}^{h,L}(Y)}{m_\gamma^{h,L}} \leq \frac{c(X,\gamma,f,L)}{m_\gamma^{h,L}}.
	\]
	Taking $\liminf_{L \to \infty}$ on both sides of this inequalty and using  Corollary \ref{cor:horoball_equid} we deduce
	\[
	\int_{\mathcal{M}_{g,n}} \eta_\epsilon(Y) \ \frac{B(Y) \cdot d\widehat{\mu}_{\text{wp}}(Y)}{b_{g,n}} \leq \liminf_{L \to \infty} \frac{c(X,\gamma,f,L)}{m_\gamma^{h,L}}.
	\]
	As 
	\[
	r(\gamma,f_\epsilon^{\min}) = r(\gamma,h) = \lim_{L \to \infty} \frac{m_\gamma^{h,L}}{L^{6g-6+2n}},
	\]
	it follows that
	\begin{equation}
	\label{eq:coun_ineq_1}
	r(\gamma,f_\epsilon^{\min}) \cdot \int_{\mathcal{M}_{g,n}} \eta_\epsilon(Y) \ \frac{B(Y) \cdot d\widehat{\mu}_{\text{wp}}(Y)}{b_{g,n}} \leq \liminf_{L \to \infty} \frac{c(X,\gamma,f,L)}{L^{6g-6+2n}}.
	\end{equation}
	Recall that $f_\epsilon^{\min} \to f$ uniformly on $(\mathbf{R}_{\geq 0})^k$ as $\epsilon \to 0$. In particular,
	\begin{align*}
	\lim_{\epsilon \to 0}  r(\gamma,f_\epsilon^{\min}) &= \lim_{\epsilon \to 0}  \int_{\mathbf{R}^k} f_\epsilon^{\min}(\mathbf{x}) \cdot W_{g,n}(\gamma,\mathbf{x}) \ \cdot d \mathbf{x}\\
	&= \int_{\mathbf{R}^k} f(\mathbf{x}) \cdot W_{g,n}(\gamma,\mathbf{x}) \cdot d\mathbf{x}\\
	&= r(f,\gamma).
	\end{align*}
	Using the properties of the functions $\eta_\epsilon \colon \mathcal{M}_{g,n} \to \mathbf{R}_{\geq0}$ one can check that
	\[
	\lim_{\epsilon \to 0}  \int_{\mathcal{M}_{g,n}} \eta_\epsilon(Y) \ \frac{B(Y) \cdot d\widehat{\mu}_{\text{wp}}(Y)}{b_{g,n}} = \frac{B(X)}{b_{g,n}}.
	\]
	Taking $\epsilon \to 0$ in (\ref{eq:coun_ineq_1}) we deduce
	\[
	r(\gamma, f) \cdot \frac{B(X)}{b_{g,n}} \leq \liminf_{L \to \infty} \frac{c(X,\gamma,f,L)}{L^{6g-6+2n}}.
	\]
	This finishes the proof of (\ref{eq:count_lb}).\\
	
	Analogous arguments using $f_\epsilon^{\max}$ instead of $f_\epsilon^\text{min}$ and (\ref{eq:count_up}) instead of (\ref{eq:count_low}) yield a proof of (\ref{eq:count_ub}). This finishes the proof of Theorem \ref{theo:length_spec_simple_red}.
\end{proof}
$ $

\begin{remark}
	\label{rem:effective_def}
	Let $\|\cdot\|_{\mathcal{C}^1}$ denote the $\mathcal{C}^1$ norm for real valued, smooth, compactly supported functions on $\mathcal{M}_{g,n}$. Carefully following the steps of the proof of Theorem \ref{theo:length_spec_simple_red}, one can check that the same methods would yield an effective version of the theorem with a power saving error term under the following polynomial equidistribution condition: There exist constants $C > 0$, $\kappa > 0$, and $\epsilon_0 > 0$ such that for every smooth, compactly supported function $\eta \colon \mathcal{M}_{g,n} \to \mathbf{R}_{\geq 0}$ and every $L > 0$,
	\[
	\bigg\vert \int_{\mathcal{M}_{g,n}} \eta(Y) \ \frac{d\widehat{\mu}_\gamma^{h,L}(Y)}{m_\gamma^{h,L}}  - \int_{\mathcal{M}_{g,n}} \eta(Y) \ \frac{B(Y) \cdot d\widehat{\mu}_{\text{wp}}(Y)}{b_{g,n}}\bigg\vert \leq C \cdot \|\eta\|_{\mathcal{C}^1} \cdot L^{-\kappa},
	\]
	where $h$ ranges over all the functions $f_\epsilon^{\min}, f_\epsilon^{\max}$ with $0 < \epsilon < \epsilon_0$.
\end{remark}
$ $

\textit{Proof of Theorem \ref{theo:length_proj_spec_simple}.} We now briefly explain how to adapt the arguments in the proof of Theorem \ref{theo:length_spec_simple} to obtain a proof of Theorem \ref{theo:length_proj_spec_simple}.\\

Let $f \colon (\mathbf{R}_{\geq0})^k \to \mathbf{R}_{\geq 0}$ and $g \colon P\mathcal{ML}_{g,n} \to \mathbf{R}_{\geq 0}$ be non-negative, continuous, compactly supported functions. For every $Y \in \mathcal{T}_{g,n}$ and every $L > 0$ consider the counting function
\begin{align}
\label{eq:count_fn_proj}
c(Y,\gamma,f,L,\mathbf{a},g) &:= \int_{\mathbf{R}^k\times P\mathcal{ML}_{g,n}} f(\mathbf{x}) \cdot g\left(\overline{\lambda}\right)  \ d\nu_{\gamma,Y,\mathbf{a}}^L\left(\mathbf{x},\overline{\lambda}\right) \\
&= \sum_{\alpha \in \text{Mod}_{g,n} \cdot \gamma} f\left( \textstyle\frac{1}{L} \cdot \vec{\ell}_\alpha(Y)\right) \cdot g\left(\overline{\mathbf{a}\cdot \alpha}\right). \nonumber
\end{align}
These counting functions depend on the marking of $Y \in \mathcal{T}_{g,n}$. Using the Stone-Weierstrass theorem, one can check that Theorem \ref{theo:length_proj_spec_simple} is equivalent to the following analogue of Theorem \ref{theo:length_spec_simple_red}.\\

\begin{theorem}
	\label{theo:length_proj_spec_simple_red}
	Let $f \colon (\mathbf{R}_{\geq0})^k \to \mathbf{R}_{\geq 0}$ and $g \colon P\mathcal{ML}_{g,n} \to \mathbf{R}_{\geq 0}$ be non-negative, continuous, compactly supported functions. Then,
	\[
	\lim_{L \to \infty} \frac{c(X,\gamma,f,L,\mathbf{a},g)}{L^{6g-6+2n}} = \frac{1}{b_{g,n}} \cdot \int_{\mathbf{R}^k} f(\mathbf{x}) \cdot W_{g,n}(\gamma, \mathbf{x}) \ d\mathbf{x} \cdot \int_{P^1\mathcal{ML}_{g,n}} g\left(\overline{\lambda}\right) \ d\mu_{\text{Thu}}^X\left(\overline{\lambda}\right)
	\]
\end{theorem}
$ $

We now explain how to adapt the techniques used in the proof Theorem \ref{theo:length_spec_simple_red} to prove Theorem \ref{theo:length_proj_spec_simple_red} . For the rest of this discussion we fix a pair of non-negative, continuous, compactly supported functions $f \colon (\mathbf{R}_{\geq0})^k \to \mathbf{R}_{\geq 0}$ and $g \colon P\mathcal{ML}_{g,n} \to \mathbf{R}_{\geq 0}$, and consider the identifications
\begin{align*}
P^1\mathcal{T}_{g,n} &= \mathcal{T}_{g,n} \times P\mathcal{ML}_{g,n}, \\
P^1\mathcal{M}_{g,n} &= (\mathcal{T}_{g,n} \times P\mathcal{ML}_{g,n}) / \text{Mod}_{g,n}.
\end{align*}
It will be important to make a clear distinction between points $Y \in \mathcal{T}_{g,n}$ and their images $[Y] := \pi(Y) \in \mathcal{M}_{g,n}$ under the quotient map $\pi \colon \mathcal{M}_{g,n} := \mathcal{T}_{g,n} / \text{Mod}_{g,n}$, as well as between points $(Y,\overline{\lambda}) \in P^1\mathcal{ML}_{g,n}$ and their images $[Y,\overline{\lambda}] \in P^1\mathcal{M}_{g,n}$ under the quotient map $\Pi \colon P^1 \mathcal{T}_{g,n} \to P^1 \mathcal{M}_{g,n}$.\\

To deal with the fact that the counting functions defined in (\ref{eq:count_fn_proj}) depend on the marking of $Y \in \mathcal{T}_{g,n}$, we introduce a local averaging procedure to obtain well define counting functions on $\mathcal{M}_{g,n}$. Using the proper discontinuity of the $\text{Mod}_{g,n}$-action on $\mathcal{T}_{g,n}$ one can find a neighborhood $W_X \subseteq \mathcal{T}_{g,n}$ of $X$ such that
\begin{enumerate}
	\item $W_X$ is $\text{Stab}(X)$-invariant,
	\item $\phi \cdot W_X \cap W_X = \emptyset$ for all $\phi \in \text{Mod}_{g,n}\setminus\text{Stab}(X)$.
\end{enumerate}
For every non-negative, continuous, compactly supported function $h \colon (\mathbf{R}_{\geq0})^k \to \mathbf{R}_{\geq 0}$, every $[Y] \in \pi(W_X)$, and every $L > 0$, consider the counting function
\[
c'\left([Y],\gamma,h,L,\mathbf{a},g\right) := \frac{1}{|\text{Stab}(X)|} \cdot \sum_{\phi \in \text{Stab{(X)}}} c\left(\phi \cdot Y,\gamma,h,L,\mathbf{a},g\right).
\]
Notice that
\[
c'\left([X],\gamma,h,L,\mathbf{a},g\right) = c\left(X,\gamma,h,L,\mathbf{a},g\right).
\]
$ $

Let $\epsilon_0:=\epsilon_0(X) > 0$ be small enough so that $U_X(\epsilon) \subseteq W_X$ for every $0 < \epsilon < \epsilon_0$. Given any $0<\epsilon < \epsilon_0$, any $Y \in U_X(\epsilon)$, and any $L > 0$, (\ref{eq:thurston_met}) ensures the following analogue of (\ref{eq:count_comparison}) holds:
\begin{equation}
\label{eq:count_comparison_2}
c'\left([Y],\gamma,f_\epsilon^{\min},L,\mathbf{a},g\right) \leq c\left(X,\gamma,f,L,\mathbf{a},g\right) \leq c'\left([Y],\gamma,f_\epsilon^{\max},L,\mathbf{a},g\right).
\end{equation}
Consider the functions $\eta_\epsilon \colon \mathcal{M}_{g,n} \to \mathbf{R}_{\geq 0}$ introduced in the proof of Theorem \ref{theo:length_spec_simple_red}. For every $0 < \epsilon < \epsilon_0$, multiplying (\ref{eq:count_comparison_2}) by $\eta_\epsilon([Y])$ and integrating over $\mathcal{M}_{g,n}$ with respect to $d\widehat{\mu}_{\text{wp}}([Y])$ yields the following analogues of (\ref{eq:count_low}) and (\ref{eq:count_up}):
\begin{equation}
\label{eq:count_low_2}
\int_{\mathcal{M}_{g,n}} \eta_\epsilon\left([Y]\right) \cdot c'\left([Y],\gamma,f_\epsilon^{\min},L,\mathbf{a},g\right) \ d\widehat{\mu}_{\text{wp}}\left([Y]\right) \leq c\left(X,\gamma,f,L,\mathbf{a},g\right),
\end{equation}
\begin{equation}
\label{eq:count_up_2}
c(X,\gamma,f,L,\mathbf{a},g) \leq \int_{\mathcal{M}_{g,n}} \eta_\epsilon\left([Y]\right) \cdot c\left([Y],\gamma,f_\epsilon^{\max},L,\mathbf{a},g\right) \ d\widehat{\mu}_{\text{wp}}\left([Y]\right).
\end{equation}
$ $

Let $p \colon P^1\mathcal{T}_{g,n} = \mathcal{T}_{g,n} \times P\mathcal{ML}_{g,n} \to P\mathcal{ML}_{g,n} $ be the map that projects to the second coordinate. Consider the function $g' \colon P^1 \mathcal{M}_{g,n} \to \mathbf{R}_{\geq 0}$ which to every $[Y,\overline{\lambda}] \in  P^1 \mathcal{M}_{g,n} $ assigns the value
\[
g' \left(\left[Y,\overline{\lambda}\right]\right) := \mathbbm{1}_{\pi(W_X)}\left([Y]\right) \cdot \frac{1}{|\text{Stab}(X)|} \cdot \sum_{\phi \in \text{Stab}(X)} g\left(\phi \cdot p \left(\Pi|_{W_X \times P\mathcal{ML}_{g,n}}^{-1}\left(\left[Y,\overline{\lambda}\right]\right)\right)\right),
\]
where $\Pi|_{W_X \times P\mathcal{ML}_{g,n}}^{-1}([Y,\overline{\lambda}]) \in W_X \times P\mathcal{ML}_{g,n}$ denotes any of the finitely many preimages of $[Y,\overline{\lambda}]$ under the restriction $\Pi|_{W_X \times P\mathcal{ML}_{g,n}}$. The following analogue of Proposition \ref{prop:pull_push} can be proved using a similar, albeit more complicated, unfolding argument.\\

\begin{proposition}
	\label{prop:pull_push_2}
	Let $h \colon (\mathbf{R}_{\geq0})^k \to \mathbf{R}_{\geq 0}$ be a non-negative, continuous, compactly supported function. Then, for every $0 < \epsilon < \epsilon_0$ and every $L > 0$,
	\begin{align*}
	&\int_{\mathcal{M}_{g,n}} \eta_\epsilon\left([Y]\right) \cdot c'\left([Y],\gamma,h,L,\mathbf{a},g\right) \ d\widehat{\mu}_{\text{wp}}\left([Y]\right) \\
	&= \int_{P^1\mathcal{M}_{g,n}} \eta_\epsilon\left([Y]\right) \cdot g'\left(\left[Y,\overline{\lambda}\right]\right) \ d\widehat{\nu}_{\gamma,\mathbf{a}}^{h,L}\left(\left[Y,\overline{\lambda}\right]\right).
	\end{align*}
\end{proposition}
$ $

Theorem \ref{theo:length_proj_spec_simple_red} can now be proved by mimicking the proof of Theorem \ref{theo:length_spec_simple_red} presented above: the inequalities (\ref{eq:count_low_2}) and (\ref{eq:count_up_2}) are used in place of the inequalities (\ref{eq:count_low}) and (\ref{eq:count_up}), Proposition \ref{prop:pull_push_2} is used in place of Proposition \ref{prop:pull_push}, and Theorem \ref{theo:horoball_equid} is used in place of Corollary \ref{cor:horoball_equid}.\\

\begin{remark}
	\label{rem:effective_def_2}
	A polynomial equidistribution condition analogous to the one introduced in Remark \ref{rem:effective_def} but for horoball segment measures on $P^1\mathcal{M}_{g,n}$ would yield an effective version of Theorem \ref{theo:length_proj_spec_simple_red} with a  power saving error term.
\end{remark}
$ $

\section{Topological factor of asymptotic length spectrum}

$ $

\textit{Setting.} For the rest of this section, let $\gamma := (\gamma_1,\dots,\gamma_k)$ with $1 \leq k \leq 3g-3+n$ be a fixed ordered simple closed multi-curve on $S_{g,n}$.\\

\textit{Proof of Theorem \ref{theo:top_int_asymp_spec}.}  By Carathéodory's extension theorem, to prove Theorem \ref{theo:top_int_asymp_spec}, it is enough to show that the measures $W_{g,n}(\gamma,\mathbf{x}) \cdot d\mathbf{x}$ and $(\widetilde{I}_{\gamma})_* (\widetilde{\mu}_{\text{Thu}}^\gamma)$ coincide on a semi-ring of subsets that generates the Borel $\sigma$-algebra of $(\mathbf{R}_{\geq 0})^k$. Consider the generating semi-ring of boxes
\[
B_{\mathbf{a},\mathbf{b}} := \prod_{i=1}^k [a_i,b_i)
\]
with $\mathbf{a} := (a_i)_{i=1}^k, \mathbf{b} := (b_i)_{i=1}^k \in (\mathbf{R}_{\geq 0})^k$ arbitrary. By the inclusion-exclusion principle and Lemma \ref{lem:thu_meas_zero}, it is enough to consider closed boxes 
\[
B_{\mathbf{b}} := \prod_{i=1}^k [0,b_i]
\]
with $\mathbf{b} := (b_i)_{i=1}^k \in (\mathbf{R}_{>0})^k$ arbitrary. \\

By Proposition \ref{prop:total_hor_meas}, 
\[
\int_{B_{\mathbf{b}}} W_{g,n}(\gamma,\mathbf{x}) \cdot d\mathbf{x} = \lim_{L \to \infty} \frac{m_\gamma^{f_\mathbf{b},L}}{L^{6g-6+2n}},
\]
where  $f_\mathbf{b} \colon (\mathbf{R}_{\geq 0})^k \to \mathbf{R}_{\geq 0}$ is the function which to every $\mathbf{x} := (x_i)_{i=1}^k \in (\mathbf{R}_{\geq 0})^k$ assigns the value
\[
f_\mathbf{b}(\mathbf{x}) := \prod_{i=1}^k \mathbbm{1}_{[0,b_i]}(x_i).
\]
By definition,
\[
m_\gamma^{f_\mathbf{b},L} := \widehat{\mu}_{\gamma}^{f_\mathbf{b},L}(\mathcal{M}_{g,n}).
\]
As $\widehat{\mu}_{\gamma}^{f_\mathbf{b},L}$ is the pushforward to $\mathcal{M}_{g,n}$ of the measure $\widetilde{\mu}_{\gamma}^{f_{\mathbf{b}},L}$ on $\mathcal{T}_{g,n}/ \text{Stab}(\gamma)$,
\[
\widehat{\mu}_{\gamma}^{f_\mathbf{b},L}(\mathcal{M}_{g,n}) =  \widetilde{\mu}_{\gamma}^{f_\mathbf{b},L}(\mathcal{T}_{g,n}/ \text{Stab}(\gamma)).
\]
Notice that
\[
d\widetilde{\mu}_{\gamma}^{f_\mathbf{b},L}(X) = f_\mathbf{b}\left(\textstyle\frac{1}{L} \cdot \vec{\ell}_\gamma(X)\right) d\widetilde{\mu}_{\text{wp}}^\gamma(X),
\]
where $\widetilde{\mu}_{\text{wp}}^\gamma$ is the local pushforward of the Weil-Petersson measure $\mu_{\text{wp}}$ on $\mathcal{T}_{g,n}$ to the quotient $\mathcal{T}_{g,n}/\text{Stab}(\gamma)$. It particular,
\[
m_\gamma^{f_\mathbf{b},L} = \widetilde{\mu}_{\text{wp}}^\gamma\left(\{ X \in \mathcal{T}_{g,n}/\text{Stab}(\gamma) \ | \ \ell_{\gamma_i}(X) \leq b_i L, \ \forall i= 1,\dots,k\} \right).
\]
$ $

It follows that, to prove Theorem \ref{theo:top_int_asymp_spec}, it is enough to prove the following result.\\

\begin{proposition}
	\label{prop:topo_dist_red}
	For any $\mathbf{b} := (b_1,\dots,b_k) \in (\mathbf{R}_{> 0})^k$,
	\begin{align*}
	&\lim_{L \to \infty} \frac{\widetilde{\mu}_{\text{wp}}^\gamma \left(\{X \in \mathcal{T}_{g,n}/\text{Stab}(\gamma) \ | \ \ell_{\gamma_i}(X) \leq b_i L, \ \forall i= 1,\dots,k\} \right)}{L^{6g-6+2n}} \\
	&= \widetilde{\mu}_{\text{Thu}}^\gamma \left( \{\lambda \in \mathcal{ML}_{g,n}/\text{Stab}(\gamma) \ | \ i(\lambda,\gamma_i) \leq b_i, \ \forall i =1,\dots,k \}\right).
	\end{align*}
\end{proposition}
$ $

Some of the arguments in our proof of Proposition \ref{prop:topo_dist_red} are closely related to ideas in the proofs of \cite[Theorem 5.17]{Mir04} and \cite[Theorem 3.3]{RS19}. The Yamabe space $\mathcal{Y}_{g,n}$, the enlarged Yamabe space $\overline{\mathcal{Y}_{g,n}}$, and Thurston's shear coordinates will play a crucial role in our proof; we refer the reader to \S 2 for definitions. \\

\textit{Shear coordinates of the enlarged Yamabe space.} Let $\mu$ be a maximal geodesic lamination on $S_{g,n}$ and $F_\mu \colon \mathcal{T}_{g,n} \to \mathcal{ML}_{g,n}(\mu)$ be the shear coordinates of $\mathcal{T}_{g,n}$ with respect to $\mu$. Consider the map $\Phi_\mu \colon \mathcal{Y}_{g,n} \to (0,\infty) \times \mathcal{ML}_{g,n}(\mu)$ given by
\[
\Phi_\mu(t,X) := (t, t \cdot F_\mu(X))
\]
for every $X \in \mathcal{T}_{g,n}$ and every $t > 0$. Using Lemma \ref{lem:asymp_shear} one can check that this map extends to a homeomorphism 
\[
\overline{\Phi_\mu} \colon \overline{\mathcal{Y}_{g,n}} \to \left((0,\infty) \times \mathcal{ML}_{g,n}(\mu)\right) \sqcup \left(\{0\} \times \mathcal{ML}_{g,n}\right),
\]
where the topology on the target comes from its natural embedding in $[0,\infty) \times \mathcal{ML}_{g,n}$, such that
\begin{equation}
\label{eq:id_ext}
\overline{\Phi}_\mu(\lambda) = (0,\lambda)
\end{equation}
for every $\lambda \in \mathcal{ML}_{g,n}$. We refer to this map as the \textit{shear coordinates} of $\overline{\mathcal{Y}_{g,n}}$ with respect to $\mu$. \\

\textit{The asymptotics of the Weil-Petersson measure.} Given $t > 0$, let $\mu_{\text{wp}}^t$ be the pushforward to $\{t\} \times \mathcal{T}_{g,n} \subseteq \overline{\mathcal{Y}_{g,n}}$ of the Weil-Petersson measure $\mu_{\text{wp}}$ on $\mathcal{T}_{g,n}$ with respect to the map 
\[
\mathcal{T}_{g,n} \to \{t\} \times \mathcal{T}_{g,n}, \quad X \mapsto (t,X).
\]
We will also denote by $\mu_\text{wp}^t$ the extension by zero of this measure to $\overline{\mathcal{Y}_{g,n}}$ and by $\mu_\text{Thu}$ the extension by zero of the measure $\mu_\text{Thu}$ on $\mathcal{ML}_{g,n} \subseteq \overline{\mathcal{Y}_{g,n}}$ to $\overline{\mathcal{Y}_{g,n}}$. The following proposition describes the asymptotic behavior of the measures $\mu_{\text{wp}}^t$ on $\overline{\mathcal{Y}_{g,n}}$ as $t \to 0$.\\

\begin{proposition}
	\label{prop:asymp_wp_meas}
	In the weak-$\star$ topology for measures on $\overline{\mathcal{Y}_{g,n}}$,
	\[
	\lim_{t \to 0} \ t^{6g-6+2n} \cdot \mu_\text{wp}^t = \mu_{\text{Thu}}.
	\]
\end{proposition}
$ $

\begin{proof}
	Let $\mu$ be a maximal geodesic lamination on $S_{g,n}$ satisfying the assumptions in the statement of Corollary \ref{cor:shear_mp} and 
	\[
	\overline{\Phi_\mu} \colon \overline{\mathcal{Y}_{g,n}} \to \left((0,\infty) \times \mathcal{ML}_{g,n}(\mu)\right) \sqcup \left(\{0\} \times \mathcal{ML}_{g,n}\right)
	\]
	be the shear coordinates of $\overline{\mathcal{Y}_{g,n}}$ with respect to $\mu$. For every $t \geq 0$ consider the measure $\mu_{\text{Thu}}^t$ on 
	\[
	\left((0,\infty) \times \mathcal{ML}_{g,n}(\mu)\right) \sqcup \left(\{0\} \times \mathcal{ML}_{g,n}\right)
	\]
	given by
	\[
	\mu_{\text{Thu}}^t := \delta_t \otimes \mu_{\text{Thu}}|_{\mathcal{ML}_{g,n}(\mu)}.
	\]
	Notice that
	\[
	\lim_{t \to 0} \ \mu_{\text{Thu}}^t = \mu_{\text{Thu}}^0
	\]
	in the weak-$\star$ topology. Using Corollary \ref{cor:shear_mp} and the scaling property of the Thurston measure one can check that, for every $t > 0$,
	\[
	\left(\overline{\Phi_\mu}\right)_*  \mu_\text{wp}^t = t^{-(6g-6+2n)} \cdot \mu_{\text{Thu}}^t.
	\]
	As the subset $\mathcal{ML}_{g,n}(\mu) \subseteq \mathcal{ML}_{g,n}$ has full measure,
	\[
	\label{eq:full_meas_ext}
	\mu_{\text{Thu}}^0 = \delta_0 \otimes \mu_{\text{Thu}}.
	\]
	This together with (\ref{eq:id_ext}) imply
	\[
	\left(\overline{\Phi_\mu}\right)_* \mu_{\text{Thu}} = \mu_\text{Thu}^0.
	\]
	Putting everything together we deduce
	\[
	\lim_{t \to 0} \ t^{6g-6+2n} \cdot \mu_\text{wp}^t = \mu_{\text{Thu}}.
	\]
	This finishes the proof.
\end{proof}
$ $

\textit{Proof of Proposition \ref{prop:topo_dist_red}.} By Proposition \ref{prop:pd_Y}, the subgroup $\text{Stab}(\gamma) \subseteq \text{Mod}_{g,n}$ acts properly discontinuously on 
\[
\overline{\mathcal{Y}_{g,n}}(\gamma) := \mathcal{Y}_{g,n} \sqcup \mathcal{ML}_{g,n}(\gamma).
\]
Let $\widetilde{\mu}_{\text{wp}}^{\gamma,t}$ and $\widetilde{\mu}_{\text{Thu}}^\gamma$ be the local pushforwards of the measures $\mu_\text{wp}^t$ and $\mu_{\text{Thu}}|_{\mathcal{ML}_{g,n}(\gamma)}$ on $\overline{\mathcal{Y}_{g,n}}(\gamma)$ to the quotient $\overline{\mathcal{Y}_{g,n}}(\gamma)/\text{Stab}(\gamma)$. Directly from Proposition \ref{prop:asymp_wp_meas} we obtain the following corollary.\\

\begin{corollary}
	\label{cor:asymp_wp_meas}
	In the weak-$\star$ topology for measures on $\overline{\mathcal{Y}_{g,n}}(\gamma)/\text{Stab}(\gamma)$,
	\[
	\lim_{t \to 0} \ t^{6g-6+2n} \cdot \widetilde{\mu}_\text{wp}^{\gamma,t} = \widetilde{\mu}_{\text{Thu}}^\gamma.
	\]
\end{corollary}
$ $

Consider the subsets
\begin{align*}
\mathcal{Y}_{g,n}^1&:= (0,1] \cdot \mathcal{T}_{g,n} \subseteq \mathcal{Y}_{g,n},\\
\overline{\mathcal{Y}_{g,n}^1} &:= \mathcal{Y}_{g,n}^1 \sqcup \mathcal{ML}_{g,n} \subseteq \overline{\mathcal{Y}_{g,n}},\\
\overline{\mathcal{Y}_{g,n}^1}(\gamma) &:= \mathcal{Y}_{g,n}^1 \sqcup \mathcal{ML}_{g,n}(\gamma) \subseteq \overline{\mathcal{Y}_{g,n}}(\gamma).
\end{align*}
Notice that $\text{Stab}(\gamma)$ preserves $\overline{\mathcal{Y}_{g,n}^1}(\gamma)  \subseteq \overline{\mathcal{Y}_{g,n}}(\gamma)$. Consider the embedded quotient
\[
\overline{\mathcal{Y}_{g,n}^1}(\gamma) /\text{Stab}(\gamma) \subseteq \overline{\mathcal{Y}_{g,n}}(\gamma) /\text{Stab}(\gamma).
\]
Given $\mathbf{b} := (b_1,\dots,b_k) \in (\mathbf{R}_{>0})^k$ let $ \widetilde{B}_\mathbf{b}(\gamma) \subseteq \overline{\mathcal{Y}_{g,n}}(\gamma)/\text{Stab}(\gamma)$ be the subset 
\[
\widetilde{B}_\mathbf{b}(\gamma) := \{\alpha \in \overline{\mathcal{Y}_{g,n}^1}(\gamma)/\text{Stab}(\gamma) \ | \ i(\alpha,\gamma_i) < b_i, \ \forall i=1,\dots,k\}.
\]
One would like to use Corollary \ref{cor:asymp_wp_meas} together with Portmanteau's theorem to deduce
\begin{equation}
\label{eq:conv_wish}
\lim_{t \to 0} \ t^{6g-6+2n} \cdot \widetilde{\mu}_\text{wp}^{\gamma,t} \left( \widetilde{B}_\mathbf{b}(\gamma) \right)= \widetilde{\mu}_{\text{Thu}}^\gamma \left( \widetilde{B}_\mathbf{b}(\gamma) \right).
\end{equation}
Notice that for every $0 < t \leq 1$,
\[
\widetilde{\mu}_\text{wp}^{\gamma,t} \left( \widetilde{B}_\mathbf{b}(\gamma) \right) = \widetilde{\mu}_{\text{wp}}^\gamma\left(\{X \in \mathcal{T}_{g,n}/\text{Stab}(\gamma) \ | \ \ell_{\alpha_i}(X) < b_i/t, \ \forall i= 1,\dots,k\} \right),
\]
and that
\[
\widetilde{\mu}_{\text{Thu}}^\gamma \left( \widetilde{B}_\mathbf{b}(\gamma) \right) = \widetilde{\mu}_{\text{Thu}}^\gamma \left( \{\lambda \in \mathcal{ML}_{g,n}/\text{Stab}(\gamma) \ | \ i(\lambda,\gamma_i) < b_i, \ \forall i =1,\dots,k \}\right).
\]
Letting $t = 1/L$ with $0 < L \leq 1$ and taking $L \searrow 0$ would prove Proposition \ref{prop:topo_dist_red}. But the hypothesis of Portmanteau's theorem are not verified by the subset $\widetilde{B}_\mathbf{b}(\gamma) \subseteq \overline{\mathcal{Y}}_{g,n}(\gamma)/\text{Stab}(\gamma)$ as it does not have compact closure. Such non-compactness comes from the fact that $\mathcal{ML}_{g,n}(\gamma) \subseteq \mathcal{ML}_{g,n}$ is open. To overcome this difficulty we will prove the following no escape of mass result.\\

\begin{proposition}
	\label{prop:no_escape_of_mass}
	Let $\mathbf{b}:=(b_1,\dots,b_k) \in (\mathbf{R}_{>0})^k$. For every $\epsilon > 0$ there exists a compact subset $\widetilde{K}_\mathbf{b}^\epsilon(\gamma)\subseteq \widetilde{B}_\mathbf{b}(\gamma)$ with the following properties:
	\begin{enumerate}
		\item $\widetilde{\mu}_\text{Thu}^\gamma \left(\partial \widetilde{K}_\mathbf{b}^\epsilon(\gamma)\right) = 0$,
		\item $\widetilde{\mu}_\text{Thu} ^\gamma \left(\widetilde{B}_\mathbf{b}(\gamma) \backslash \widetilde{K}_\mathbf{b}^\epsilon(\gamma)\right) < \epsilon$,
		\item $t^{6g-6+2n} \cdot \widetilde{\mu}_\text{wp}^{\gamma,t}\left(\widetilde{B}_\mathbf{b}(\gamma) \backslash \widetilde{K}_\mathbf{b}^\epsilon(\gamma)\right) < \epsilon$ for all small enough $t > 0$.
	\end{enumerate}
\end{proposition}
$ $

Let us prove Proposition \ref{prop:topo_dist_red} using Proposition \ref{prop:no_escape_of_mass}. \\

\begin{proof}[Proof of Proposition \ref{prop:topo_dist_red}]
	Following the discussion above, it remains to verify (\ref{eq:conv_wish}). Fix $\mathbf{b}:=(b_1,\dots,b_k) \in (\mathbf{R}_{>0})^k$ and let $\epsilon > 0$ be arbitrary. Consider the compact subset $\widetilde{K}_\mathbf{b}^\epsilon(\gamma)\subseteq \widetilde{B}_\mathbf{b}(\gamma)$ given by Proposition \ref{prop:no_escape_of_mass}. As $\widetilde{K}_\mathbf{b}^\epsilon(\gamma)\subseteq \overline{\mathcal{Y}}_{g,n}(\gamma)/\text{Stab}(\gamma)$ is compact and satisfies $\widetilde{\mu}_\text{Thu}^\gamma(\partial \widetilde{K}_\mathbf{b}^\epsilon(\gamma)) = 0$, Corollary \ref{cor:asymp_wp_meas} together with Portmanteau's theorem imply
	\[
	\lim_{t \to 0} \ t^{6g-6+2n} \cdot \widetilde{\mu}_\text{wp}^{\gamma,t} \left( \widetilde{K}_\mathbf{b}^\epsilon(\gamma)\right)= \widetilde{\mu}_{\text{Thu}}^\gamma \left(\widetilde{K}_\mathbf{b}^\epsilon(\gamma)\right).
	\] 
	Let $t_0 > 0$ be small enough so that 
	\[
	\bigg\vert t^{6g-6+2n} \cdot \widetilde{\mu}_\text{wp}^{\gamma,t} \left( \widetilde{K}_\mathbf{b}^\epsilon(\gamma)\right) - \widetilde{\mu}_{\text{Thu}}^\gamma \left(\widetilde{K}_\mathbf{b}^\epsilon(\gamma)\right) \bigg\vert < \epsilon
	\]
	and 
	\[
	t^{6g-6+2n} \cdot \widetilde{\mu}_\text{wp}^{\gamma,t}\left(\widetilde{B}_\mathbf{b}(\gamma) \backslash \widetilde{K}_\mathbf{b}^\epsilon(\gamma)\right) < \epsilon
	\]
	for every $0 < t < t_0$. As $\widetilde{\mu}_\text{Thu}^\gamma(\widetilde{B}_\mathbf{b}(\gamma) \backslash \widetilde{K}_\mathbf{b}^\epsilon(\gamma)) < \epsilon$, the triangle inequality implies
	\[
	\bigg\vert t^{6g-6+2n} \cdot \widetilde{\mu}_\text{wp}^{\gamma,t} \left( \widetilde{B}_\mathbf{b}(\gamma) \right) - \widetilde{\mu}_{\text{Thu}}^\gamma \left( \widetilde{B}_\mathbf{b}(\gamma) \right)\bigg\vert < 3 \epsilon
	\]
	for every $0 < t < t_0$. As $\epsilon > 0$ is arbitrary, this proves (\ref{eq:conv_wish}) and thus concludes the proof of Proposition \ref{prop:topo_dist_red}.
\end{proof}
$ $

The rest of this section is devoted to proving Proposition \ref{prop:no_escape_of_mass}. To define the compact subsets $\widetilde{K}_\mathbf{b}^\epsilon(\gamma)\subseteq \widetilde{B}_\mathbf{b}(\gamma)$ we approximate the open condition $\lambda \in \mathcal{ML}_{g,n}(\gamma)$ by a sequence of closed conditions. \\

\textit{Filling together with a simple closed multi-curve.} Consider the subset of $\mathcal{ML}_{g,n}$,
\[
\mathcal{ML}_{g,n}^\gamma := \{\lambda \in \mathcal{ML}_{g,n} \ | \ i(\lambda,\gamma_i) = 0, \ \forall i=1,\dots,k\}.
\]
This subset is homogeneous and closed. In particular, it is projectively compact. Let $\mathcal{S}_{g,n}^\gamma\subseteq \mathcal{ML}_{g,n}$ be the subset of all simple closed curves on $S_{g,n}$ that belong to $\mathcal{ML}_{g,n}^\gamma$. This subset if discrete and closed. Notice that every component of $\gamma$ belongs to $\mathcal{S}_{g,n}^\gamma$. Consider the map $s_\gamma \colon \overline{\mathcal{Y}_{g,n}} \to \mathbf{R}_{\geq 0}$ which to every $\alpha \in \overline{\mathcal{Y}_{g,n}}$ assigns the value
\[
s_\gamma(\alpha) := \inf_{\beta \in \mathcal{S}_{g,n}^\gamma} i(\alpha,\beta).
\]
We will refer $s_\gamma(\alpha)$ as \textit{the systole of $\alpha$ relative to $\gamma$}. As complete, finite area hyperbolic surfaces always have a simple closed curve of shortest length, $s_\gamma(\alpha) > 0$ for every $\alpha \in \mathcal{Y}_{g,n}$ and the infimum defining this quantity is realized. The following proposition characterizes the subset $\mathcal{ML}_{g,n}(\gamma) \subseteq \mathcal{ML}_{g,n}$ in terms of this function.\\

\begin{proposition}
	\label{prop:fil_char}
	Given $\lambda \in \mathcal{ML}_{g,n}$,
	\[
	\lambda \in \mathcal{ML}_{g,n}(\gamma)  \ \Leftrightarrow \ s_\gamma(\lambda) > 0.
	\]
	Moreover, if $\lambda \in \mathcal{ML}_{g,n}(\gamma)$ then the infimum defining $s_\gamma(\lambda)$ is attained.
\end{proposition}
$ $

\begin{proof}
	Let us first assume that $\lambda \notin \mathcal{ML}_{g,n}(\gamma)$. By Proposition \ref{prop:ml_fil}, one can find $\eta \in \mathcal{ML}_{g,n}$ such that $i(\gamma,\eta) = i(\lambda,\eta) = 0$. If one of the components of $\gamma$ is a minimal component of $\eta $ then $s_\gamma(\lambda) = 0$. Assume then that $\eta$ has a minimal component $\eta'$ which is not one of the components of $\gamma$. Given $\epsilon > 0$, as $\eta'$ is minimal and not one of the components of $\gamma$, one can follow any half-leaf of $\eta'$ for long enough so that it comes back near to its starting point in such a way that it can be closed up by adding an arc disjoint from the components of $\gamma$ and whose tranverse measure with respect to $\lambda$ is $\leq \epsilon$. This produces a simple closed curve $\beta \in \mathcal{S}_{g,n}^\gamma$ such that $i(\lambda,\beta) \leq \epsilon$. As $\epsilon > 0$ is arbitrary, this shows that $s_\gamma(\lambda) =0$. \\
	
	We now assume that $\lambda \in \mathcal{ML}_{g,n}(\gamma)$. Consider the restriction 
	\[
	i(\lambda,\cdot)|_{\mathcal{ML}_{g,n}^\gamma} \colon \mathcal{ML}_{g,n}^\gamma \to \mathbf{R}_{\geq 0}.
	\]
	By Proposition \ref{prop:ml_fil}, this function takes only positive values. From this and the projective compactness of $\mathcal{ML}_{g,n}^\gamma$ it follows that this function is proper. As $\mathcal{S}_{g,n}^\gamma \subseteq \mathcal{ML}_{g,n}^\gamma$ is a discrete closed subset, we deduce that $s_\gamma(\lambda) > 0$ and moreover that the infimum defining this quantity is attained. This finishes the proof.
\end{proof}
$ $

One can check that the systole relative to $\gamma$ is continuous as a function on $\overline{\mathcal{Y}_{g,n}}$. We record this and other properties in the following proposition.\\

\begin{proposition}
	\label{prop:syst_map}
	The systole relative to $\gamma$,
	\[
	s_\gamma \colon \overline{\mathcal{Y}}_{g,n} \to \mathbf{R}_{\geq 0},
	\]
	is homogeneous, $\text{Stab}(\gamma)$-equivariant, and continuous. \\
\end{proposition}

\begin{proof}
	The homogenity and $\text{Stab}(\gamma)$-equivariance of $s_\gamma$
	can be checked directly from the definition. We now show that $s_\gamma$ is continuous. Consider first $\alpha \in \overline{\mathcal{Y}}_{g,n}$ such that
	$s_\gamma(\alpha) = 0$. Let $\epsilon > 0$ be arbitrary. As $s_\gamma(\alpha) = 0$, we can find $\beta \in \mathcal{S}_{g,n}^\gamma$ such that $i(\alpha,\beta) < \epsilon$. Consider the open neighborhood $U \subseteq \overline{\mathcal{Y}_{g,n}}$ of $\alpha$ given by
	\[
	U := \{\sigma \in \overline{\mathcal{Y}_{g,n}} \ | \ i(\sigma,\beta) < \epsilon \}. 
	\]
	Notice that $s_\gamma(\sigma) < \epsilon$ for every $\sigma \in U$. As $\epsilon> 0$ is arbitrary, this shows that $s_\gamma$ is continuous at every $\alpha \in \overline{\mathcal{Y}}_{g,n}$ such that
	$s_\gamma(\alpha) = 0$.\\
	
	Now consider $\alpha \in \overline{\mathcal{Y}}_{g,n}$ such that
	$s_\gamma(\alpha) > 0$. Let $1 < \epsilon < 2$ be arbitrary. Let $U' \subseteq \overline{\mathcal{Y}_{g,n}}$ be a compact neighborhood of $\alpha$. As $\mathcal{ML}_{g,n}^\gamma$ is projectively compact, one can find a constant $C > 0$ such that
	\[
	\frac{1}{C} \leq \frac{i(\beta,\lambda)}{i(\alpha,\lambda)} \leq C
	\]
	for every $\lambda \in \mathcal{ML}_{g,n}^\gamma$ and every $\beta \in U'$. In particular, if $\lambda \in \mathcal{ML}_{g,n}^\gamma$ is such that
	$i(\alpha,\lambda) > 2Cs_\gamma(\alpha)$, then $i(\beta,\lambda) > 2s_\gamma(\alpha)$ for every $\beta \in U'$. Consider the subset $K \subseteq \mathcal{ML}_{g,n}^\gamma$ given by
	\[
	K := \{\lambda \in \mathcal{ML}_{g,n}^\gamma \ | \ 	i(\alpha,\lambda) \leq 2Cs_\gamma(\alpha) \}.
	\]
	As the restriction 
	\[
	i(\alpha,\cdot)|_{\mathcal{ML}_{g,n}^\gamma} \colon \mathcal{ML}_{g,n}^\gamma \to \mathbf{R}_{>0}
	\]
	is proper (see the proof of Proposition \ref{prop:fil_char}), this set is compact. As $\mathcal{S}_{g,n}^\gamma \subseteq \mathcal{ML}_{g,n}^\gamma$ is a discrete  closed subset, $\mathcal{S}_{g,n}^\gamma \cap K$ is finite. Consider the neighborhood $U \subseteq \overline{\mathcal{Y}_{g,n}}$ of $\alpha$ given by
	\[
	U := \left\lbrace \sigma \in U' \ | \ \textstyle\frac{1}{\epsilon} \cdot i(\alpha,\beta) < i(\sigma,\beta) < \epsilon \cdot i(\alpha,\beta), \ \forall \beta \in \mathcal{S}_{g,n}^\gamma \cap K\right \rbrace.
	\]
	Notice that
	\[
	\textstyle\frac{1}{\epsilon} \cdot s_\gamma(\alpha)\leq s_\gamma(\sigma)  \leq \epsilon \cdot s_\gamma(\alpha)
	\]
	for every $\sigma \in U$. As $1 < \epsilon < 2$ is arbitrary, this shows that $s_\gamma$ is continuous at every $\alpha \in \overline{\mathcal{Y}}_{g,n}$ such that
	$s_\gamma(\alpha) > 0$. This finishes the proof.
\end{proof}
$ $

It follows from Propositions \ref{prop:fil_char} and \ref{prop:syst_map} that the restriction
\[
s_\gamma|_{\overline{\mathcal{Y}_{g,n}}(\gamma)} \colon \overline{\mathcal{Y}_{g,n}}(\gamma) \to \mathbf{R}_{>0}
\]
induces a homogeneous, positive, continuous map on the quotient $\overline{\mathcal{Y}_{g,n}}(\gamma)/\text{Stab}(\gamma)$.\\

\textit{No escape of mass.} We are now ready to introduce a family of compact subsets satisfying the properties described in Proposition \ref{prop:no_escape_of_mass}. For every $\mathbf{b} := (b_1,\dots,b_k) \in (\mathbf{R}_{>0})^k$ and every $\delta > 0$  consider the subset $\widetilde{\mathcal{K}}_\mathbf{b}^\delta(\gamma) \subseteq \widetilde{B}_\mathbf{b}(\gamma)$ given by
\[
\widetilde{\mathcal{K}}_\mathbf{b}^\delta(\gamma) := 
\left\lbrace
\begin{array}{c | l}
\alpha \in \overline{\mathcal{Y}_{g,n}^1}(\gamma) /\text{Stab}(\gamma)
& \  i(\alpha,\gamma_i) \leq b_i, \ \forall i=1,\dots,k,\\
& \ s_\gamma(\alpha) \geq \delta.\\
\end{array} \right\rbrace.
\]
$ $

\begin{proposition}
	\label{prop:K_props}
	Let $\mathbf{b}:=(b_1,\dots,b_k) \in (\mathbf{R}_{>0})^k$. The subsets $\widetilde{\mathcal{K}}_\mathbf{b}^\delta(\gamma) \subseteq \widetilde{B}_\mathbf{b}(\gamma)$ are compact and satisfy the following conditions:
	\begin{enumerate}
		\item $\widetilde{\mu}_\text{Thu}^\gamma\left(\partial \widetilde{\mathcal{K}}_\mathbf{b}^\delta(\gamma)\right) = 0$,
		\item $ \lim_{\delta \to 0 } \  \widetilde{\mu}_\text{Thu}^\gamma\left(\widetilde{B}_\mathbf{b}(\gamma) \backslash \widetilde{\mathcal{K}}_\mathbf{b}^\delta(\gamma)\right) = 0$,
		\item There exists a constant $C> 0$ such that for every $0 < \delta < 1$,
		\[
		\limsup_{t \to 0} \ t^{6g-6+2n} \cdot \widetilde{\mu}_\text{wp}^{\gamma,t}\left(\widetilde{B}_\mathbf{b}(\gamma) \backslash \widetilde{\mathcal{K}}_\mathbf{b}^\delta(\gamma)\right) \leq C \cdot \delta.
		\]
	\end{enumerate}
\end{proposition}
$ $

Proposition \ref{prop:no_escape_of_mass} follows directly from proposition \ref{prop:K_props}. For the rest of this section we fix $\mathbf{b} := (b_1,\dots,b_k) \in (\mathbf{R}_{>0})^k$ and show that the subsets $\widetilde{\mathcal{K}}_\mathbf{b}^\delta(\gamma) \subseteq \widetilde{B}_\mathbf{b}(\gamma)$ with $\delta > 0$ satisfy the conditions described in Proposition \ref{prop:K_props}.\\

\textit{Bers's theorem for $\overline{\mathcal{Y}_{g,n}^1}(\gamma)$.} Complete $\gamma$ to a maximal geodesic lamination $\mu$ of $S_{g,n}$ and consider the shear coordinates $F_\mu \colon \mathcal{T}_{g,n} \to \mathcal{ML}_{g,n}(\mu)$ of $\mathcal{T}_{g,n}$ with respect to $\mu$. Properties (\ref{eq:shear_ineq}) and (\ref{eq:shear_eq}) allow one to deduce the following analogue of Bers's theorem from Theorem \ref{theo:bers}.\\

\begin{corollary}
	\label{cor:bers_Y}
	There exists a constant $C \geq \max_{i=1,\dots,k}  b_i$ such that for any $\alpha \in \overline{\mathcal{Y}_{g,n}^1}(\gamma)$ satisfying
	\[
	i(\alpha,\gamma_i) \leq b_i, \ \forall i =1,\dots,k,
	\]
	there exists a completion $\mathcal{P} := (\gamma_1,\dots,\gamma_{3g-3+n})$ of $\gamma$ to a pair of pants decomposition of $S_{g,n}$ satisfying
	\[
	i(\alpha,\gamma_i) \leq C,  \ \forall i = 1,\dots,3g-3+n.
	\]
\end{corollary}
$ $

\textit{Compactness.} We now prove that the subsets 
\[
\widetilde{\mathcal{K}}_\mathbf{b}^\delta(\gamma) \subseteq \overline{\mathcal{Y}_{g,n}}(\gamma)/ \allowbreak \text{Stab}(\gamma)
\]
are compact. The following result is an analogue of Mumford's compactness criterion; see for instance \cite[Theorem 12.6]{FM11}.\\

\begin{proposition}
	\label{prop:K_compact}
	For every $\delta > 0$ the set $\widetilde{\mathcal{K}}_\mathbf{b}^\delta(\gamma)$ is compact.\\
\end{proposition}

\begin{proof}
	Fix $\delta > 0$. Notice that the subset $\mathcal{K}_\mathbf{b}^\delta(\gamma) \subseteq \overline{\mathcal{Y}_{g,n}}(\gamma)$ given by
	\[
	\mathcal{K}_\mathbf{b}^\delta(\gamma) := 
	\left\lbrace
	\begin{array}{c | l}
	\alpha \in \overline{\mathcal{Y}_{g,n}^1}(\gamma)
	& \  i(\alpha,\gamma_i) \leq b_i, \ \forall i=1,\dots,k,\\
	& \ s_\gamma(\alpha) \geq \delta.\\
	\end{array} \right\rbrace
	\]
	is mapped onto the subset $\widetilde{\mathcal{K}}_\mathbf{b}^\delta(\gamma) \subseteq \overline{\mathcal{Y}_{g,n}}(\gamma)/\text{Stab}(\gamma)$ by the quotient map 
	\[
	\overline{\mathcal{Y}_{g,n}}(\gamma) \to \overline{\mathcal{Y}_{g,n}}(\gamma)/\text{Stab}(\gamma).
	\]
	To prove $\widetilde{\mathcal{K}}_\mathbf{b}^\delta(\gamma) \subseteq \overline{\mathcal{Y}_{g,n}}(\gamma)/\text{Stab}(\gamma)$ is compact, it is enough to show that $\mathcal{K}_\mathbf{b}^\delta(\gamma) \subseteq \overline{\mathcal{Y}_{g,n}}(\gamma)$ can written as a finite union of $\text{Stab}(\gamma)$-orbits of compact subsets of $\overline{\mathcal{Y}_{g,n}}(\gamma)$. \\
	
	Let $C > 0$ be as in Corollary \ref{cor:bers_Y}. Notice that up to the action of $\text{Stab}(\gamma)$ there are finitely many pair of pants decompositions $\mathcal{P}$ of $S_{g,n}$ containing the components of $\gamma$. It follows from Corollary \ref{cor:bers_Y} that $\mathcal{K}_\mathbf{b}^\delta(\gamma) \subseteq \overline{\mathcal{Y}_{g,n}}(\gamma)$ can be written as the union of finitely many $\text{Stab}(\gamma)$-orbits of  subsets  $\mathcal{C}_\mathbf{b}^\delta(\mathcal{P}) \subseteq \overline{\mathcal{Y}_{g,n}}(\gamma)$ of the form
	\[
	\mathcal{C}_\mathbf{b}^\delta(\mathcal{P}) := 
	\left\lbrace
	\begin{array}{c | l}
	\alpha \in \overline{\mathcal{Y}_{g,n}^1}(\gamma)
	& \ i(\alpha,\gamma_i) \leq b_i, \ \forall i=1,\dots,k,\\
	& \ i(\alpha,\gamma_i) \leq C, \ \forall i=k+1,\dots,3g-3+n,\\
	& \ s_\gamma(\alpha) \geq \delta.
	\end{array} \right\rbrace,
	\]
	where $\mathcal{P} := (\gamma_1,\dots,\gamma_{3g-3+n})$ is a pair of pants decomposition of $S_{g,n}$ containing the components of $\gamma$. We now show that each one of the $\text{Stab}(\mathcal{P})$-invariant subsets $\mathcal{C}_\mathbf{b}^\delta(\mathcal{P}) \subseteq \overline{\mathcal{Y}_{g,n}}(\gamma)$ can be written as the $\text{Stab}(\mathcal{P})$-orbit of a compact subset of $\overline{\mathcal{Y}_{g,n}}(\gamma)$. As $\text{Stab}(\mathcal{P}) \subseteq \text{Stab}(\gamma)$, this finishes the proof. \\
	
	Fix a pair of pants decomposition $\mathcal{P} := (\gamma_1,\dots,\gamma_{3g-3+n})$ of $S_{g,n}$ containing the components of $\gamma$. By Proposition \ref{prop:fil_char}, $\mathcal{C}_\mathbf{b}^\delta(\mathcal{P}) \subseteq \overline{\mathcal{Y}_{g,n}}$ can be rewritten as
	\[
	\mathcal{C}_\mathbf{b}^\delta(\mathcal{P}) = 
	\left\lbrace
	\begin{array}{c | l}
	\alpha \in \overline{\mathcal{Y}_{g,n}^1}
	& \ i(\alpha,\gamma_i) \leq b_i, \ \forall i=1,\dots,k,\\
	& \ i(\alpha,\gamma_i) \leq C, \ \forall i=k+1,\dots,3g-3+n,\\
	& \ s_\gamma(\alpha) \geq \delta.
	\end{array} \right\rbrace.
	\]
	It follows that $\mathcal{C}_\mathbf{b}^\delta(\mathcal{P})$ is a closed (see Proposition \ref{prop:syst_map}) subset of the $\text{Stab}(\mathcal{P})$-invariant subset $\mathcal{D}_\mathbf{b}^\delta(\mathcal{P}) \subseteq \overline{\mathcal{Y}_{g,n}}$ given by
	\[
	\mathcal{D}_\mathbf{b}^\delta(\mathcal{P}) := 
	\left\lbrace
	\begin{array}{c | l}
	\alpha \in \overline{\mathcal{Y}_{g,n}^1}
	& \ \delta \leq i(\alpha,\gamma_i) \leq b_i, \ \forall i=1,\dots,k,\\
	& \ \delta \leq i(\alpha,\gamma_i) \leq C, \ \forall i=k+1,\dots,3g-3+n.\\
	\end{array} \right\rbrace.
	\]
	If we show that $\mathcal{D}_\mathbf{b}^\delta(\mathcal{P}) \subseteq \overline{\mathcal{Y}_{g,n}}$ is the $\text{Stab}(\mathcal{P})$-orbit of a compact subset $\mathcal{E}_\mathbf{b}^\delta(\mathcal{P}) \subseteq \overline{\mathcal{Y}_{g,n}}$, then $\mathcal{C}_\mathbf{b}^\delta(\mathcal{P}) \subseteq \overline{\mathcal{Y}_{g,n}}$ will be the $\text{Stab}(\mathcal{P})$-orbit of the compact subset $\mathcal{C}_\mathbf{b}^\delta(\mathcal{P}) \cap \mathcal{E}_\mathbf{b}^\delta(\mathcal{P})$, thus finishing the proof. \\
	
	Complete $\mathcal{P}$ to a maximal geodesic lamination $\mu$ of $S_{g,n}$ and consider the shear coordinates of $\overline{\mathcal{Y}_{g,n}}$ with respect to $\mu$,
	\[
	\overline{\Phi_\mu} \colon \overline{\mathcal{Y}_{g,n}} \to \left((0,\infty) \times \mathcal{ML}_{g,n}(\mu)\right) \sqcup \left(\{0\} \times \mathcal{ML}_{g,n}\right)
	\]
	By (\ref{eq:shear_eq} and (\ref{eq:id_ext}),
	\[
	i(\overline{\Phi_\mu}(\alpha),\gamma_i) = i(\alpha,\gamma_i)
	\]
	for every $\alpha \in \overline{\mathcal{Y}_{g,n}}$ and every $i=1,\dots,3g-3+n$. It follows that 
	\[
	\overline{\Phi_\mu}(\mathcal{D}_\mathbf{b}^\delta(\mathcal{P})) = [0,1] \times D_\mathbf{b}^\delta(\mathcal{P}),
	\]
	where $D_\mathbf{b}^\delta(\mathcal{P})  \subseteq \mathcal{ML}_{g,n}(\mu)$ is the subset given by
	\[
	D_\mathbf{b}^\delta(\mathcal{P}) := 
	\left\lbrace
	\begin{array}{c | l}
	\lambda \in \mathcal{ML}_{g,n}(\mu)
	& \ \delta \leq i(\alpha,\gamma_i) \leq b_i, \ \forall i=1,\dots,k,\\
	& \ \delta \leq i(\alpha,\gamma_i) \leq C, \ \forall i=k+1,\dots,3g-3+n.\\
	\end{array} \right\rbrace.
	\]
	Notice that, as $\mathcal{ML}_{g,n}(\mathcal{P}) \subseteq \mathcal{ML}_{g,n}(\mathcal{\mu})$ and as $\mathcal{P}$ is a pair of pants decomposition of $S_{g,n}$, $D_\mathbf{b}^\delta(\mathcal{P})  \subseteq \mathcal{ML}_{g,n}$ can be rewritten as
	\[
	D_\mathbf{b}^\delta(\mathcal{P}) := 
	\left\lbrace
	\begin{array}{c | l}
	\lambda \in \mathcal{ML}_{g,n}
	& \ \delta \leq i(\alpha,\gamma_i) \leq b_i, \ \forall i=1,\dots,k,\\
	& \ \delta \leq i(\alpha,\gamma_i) \leq C, \ \forall i=k+1,\dots,3g-3+n.\\
	\end{array} \right\rbrace.
	\]
	As $\overline{\Phi}_\mu$ is $\text{Stab}(\mu)$-equivariant and as the Dehn twists along the components of $\mathcal{P}$ belong to $\text{Stab}(\mu)$, it is enough for our  purposes to show that $D_\mathbf{b}^\delta(\mathcal{P})  \subseteq \mathcal{ML}_{g,n}$ can be written as the orbit of a compact subset of $\mathcal{ML}_{g,n}$ under the action of the group generated by the Dehn twists along the components of $\mathcal{P}$.\\
	
	Let $(m_i,t_i)_{i=1}^{3g-3+n}$ be a set of Dehn-Thurston coordinates of $\mathcal{ML}_{g,n}$ adapted to $\mathcal{P}$ and denote by $\Theta \subseteq (\mathbf{R}_{\geq 0} \times \mathbf{R})^{3g-3+n}$ its parameter space. Notice that $D_\mathbf{b}^\delta(\mathcal{P})  \subseteq \mathcal{ML}_{g,n}$ can be described in such coordinates as
	\[
	D_\mathbf{b}^\delta(\mathcal{P}) = 
	\left\lbrace
	\begin{array}{c | l}
	(m_i,t_i)_{i=1}^{3g-3+n} \in \Theta 
	& \ \delta \leq m_i \leq b_i , \ \forall i=1,\dots,k,\\
	& \ \delta \leq m_i \leq C, \ \forall i=k+1,\dots,3g-3+n.\\
	\end{array} \right\rbrace.
	\]
	Consider the compact subset $E_\mathbf{b}^\delta(\mathcal{P}) \subseteq \mathcal{ML}_{g,n}$ described in coordinates as
	\[
	E_\mathbf{b}^\delta(\mathcal{P}) := 
	\left\lbrace
	\begin{array}{c | l}
	(m_i,t_i)_{i=1}^{3g-3+n} \in \Theta 
	& \ \delta \leq m_i \leq b_i, \ \forall i=1,\dots,k,\\
	& \ \delta \leq m_i \leq C, \ \forall i=k+1,\dots,3g-3+n,\\
	& \ 0 \leq t_i \leq m_i, \ \forall i=1,\dots,3g-3+n.
	\end{array} \right\rbrace.
	\]
	Notice that $D_\mathbf{b}^\delta(\mathcal{P})$ is the orbit of $E_\mathbf{b}^\delta(\mathcal{P})$ under the action of the group generated by the Dehn twists along the components of $\mathcal{P}$. This finishes the proof.
\end{proof}
$ $

\textit{Measure estimates.} We now show that the subsets $\widetilde{\mathcal{K}}_\mathbf{b}^\delta(\gamma) \subseteq \overline{\mathcal{Y}_{g,n}}(\gamma)/\text{Stab}(\gamma)$ satisfy the measure estimates described by conditions (1), (2), and (3) in Proposition \ref{prop:K_props}. Condition (1) is a direct consequence of Lemma \ref{lem:thu_meas_zero} and Proposition \ref{prop:syst_map}. Notice that, as a consequence of Proposition \ref{prop:fil_char}, $\widetilde{\mathcal{K}}_\mathbf{b}^\delta(\gamma) \nearrow \widetilde{B}_\mathbf{b}(\gamma)$ as $\delta \searrow 0$. Condition (2) then follows from the continuity of the measure $\widetilde{\mu}_\text{Thu}^\gamma$ on $\overline{\mathcal{Y}_{g,n}}(\gamma)$  and the following result, which can be proved using arguments similar to the ones in the proof of Proposition \ref{prop:K_compact}.\\

\begin{lemma}
	\label{lem:finite_measure_B}
	The subset $\widetilde{B}_\mathbf{b}(\gamma) \subseteq \overline{\mathcal{Y}_{g,n}}(\gamma)/\text{Stab}(\gamma) $ has finite $\widetilde{\mu}_{\text{Thu}}^\gamma$ measure.
\end{lemma}
$ $

It remains to show that condition (3) of Proposition \ref{prop:K_props} holds.\\

\begin{proposition}
	\label{prop:meas_est_3}
	There exists $C> 0$ such that for every $0 < \delta < 1$, 
	\[
	\limsup_{t \to 0} \ t^{6g-6+2n} \cdot \widetilde{\mu}_\text{wp}^{\gamma,t}\left(\widetilde{B}_\mathbf{b}(\gamma) \backslash \widetilde{\mathcal{K}}_\mathbf{b}^\delta(\gamma)\right) \leq C \cdot \delta.
	\]
\end{proposition}
$ $

\begin{proof}
	Let $0 < \delta < 1$ be arbitrary. Notice that $\alpha \in \overline{\mathcal{Y}}_{g,n}(\gamma)/\text{Stab}(\gamma)$ belongs to $\widetilde{B}_\mathbf{b}(\gamma) \backslash \widetilde{\mathcal{K}}_\mathbf{b}^\delta(\gamma)$ if and only if
	\[
	i(\alpha,\gamma_i) \leq b_i, \ \forall i =1,\dots,k,
	\]
	and at least one of the following conditions holds:
	\begin{enumerate}
		\item $i(\alpha,\gamma_i) < \delta$ for some $i=1,\dots,k$,
		\item $i(\alpha,\beta) < \delta$ for some $\beta \in \mathcal{S}_{g,n}^\gamma$ which is not a component of $\gamma$.
	\end{enumerate}
	In particular, for every $t > 0$,
	\[
	\widetilde{\mu}_\text{wp}^{\gamma,t} \left(\widetilde{B}_\mathbf{b}(\gamma) \backslash \widetilde{\mathcal{K}}_\mathbf{b}^\delta(\gamma)\right)
	\]
	is equal to the  $\widetilde{\mu}_\text{wp}^\gamma$ measure of the set of all $X \in \mathcal{T}_{g,n}/\text{Stab}(\gamma)$ such that
	\[
	\ell_{\gamma_i}(X) \leq b_i/t, \ \forall i =1,\dots,k,
	\]
	and at least one of the following conditions holds:
	\begin{enumerate}
		\item $\ell_{\gamma_i}(X)< \delta/t$ for some $i=1,\dots,k$,
		\item $\ell_{\beta}(X) < \delta/t$ for some $\beta \in \mathcal{S}_{g,n}^\gamma$ which is not a component of $\gamma$.
	\end{enumerate}
	$ $
	
	This quantity can be estimated using Mirzakhani's integration formulas in \cite{Mir07a}. More specifically, following arguments similar to those in the proof of \cite[Proposition 3.9]{Ara20a}, one can show that, for sufficiently small $t > 0$,
	\[
	\widetilde{\mu}_\text{wp}^{\gamma,t} \left(\widetilde{B}_\mathbf{b}(\gamma) \backslash \widetilde{\mathcal{K}}_\mathbf{b}^\delta(\gamma)\right) \leq \delta \cdot P(1/t^2),
	\]
	where $P$ is a polynomial of degree $3g-3+n$ depending only on $g$, $n$, $\gamma$, and $\mathbf{b}$. It follows that
	\[	
	\limsup_{t \to 0} \ t^{6g-6+2n} \cdot \widetilde{\mu}_\text{wp}^{\gamma,t}\left(\widetilde{B}_\mathbf{b}(\gamma) \backslash \widetilde{\mathcal{K}}_\mathbf{b}^\delta(\gamma)\right) \leq C \cdot \delta
	\]
	for some constant $C > 0$ depending only on $g$, $n$, $\gamma$, and $\mathbf{b}$. This finishes the proof.
\end{proof}
$ $

\section{Counting filling closed hyperbolic multi-geodesics}

$ $

\textit{Setting.} For the rest of this section, let $\gamma := (\gamma_1,\dots,\gamma_k)$ with $k \geq 1$ be an ordered filling closed multi-curve on $S_{g,n}$ and $X \in \mathcal{T}_{g,n}$ be a marked, oriented, complete, finite area hyperbolic structure on $S_{g,n}$.\\

\textit{Proof of Theorem \ref{theo:length_spec_filling}.} As $\gamma$ is filling, its stabilizer $\text{Stab}(\gamma) \subseteq \text{Mod}_{g,n}$ is finite. Consider the family of rescaled counting measures $\{\overline{\mu}_{\gamma,X}^L\}_{L>0}$ on $(\mathbf{R}_{\geq 0})^k$ given by 
\[
\overline{\mu}_{\gamma,X}^L := \sum_{\phi \in \text{Mod}_{g,n}} \delta_{\frac{1}{L} \cdot \vec{\ell}_{\phi \cdot \gamma}(X)}.
\]
Notice that, for every $L > 0$,
\begin{align*}
\overline{\mu}_{\gamma,X}^L &= | \text{Stab}(\gamma) | \cdot \mu_{\gamma,X}^L,\\
(I_\gamma)_*(\mu_{\text{Thu}}^\gamma) &= | \text{Stab}(\gamma) | \cdot (\widetilde{I}_\gamma)_*(\widetilde{\mu}_{\text{Thu}}^\gamma).
\end{align*}
It follows that, to prove Theorem \ref{theo:length_spec_filling}, it is equivalent to show
\begin{equation}
\label{eq:length_spec_filling_reduced}
\lim_{L \to \infty} \frac{\overline{\mu}_{\gamma,X}^L}{L^{6g-6+2n}} = \frac{B(X)}{b_{g,n}} \cdot (I_\gamma)_*(\mu_{\text{Thu}}^\gamma)
\end{equation}
in the weak-$\star$ topology for measures on $(\mathbf{R}_{\geq 0})^k$.\\

By standard approximation arguments, to prove (\ref{eq:length_spec_filling_reduced}), it is equivalent to show
\[
\lim_{L \to \infty} \frac{\mu_{\gamma,X}^L(A)}{L^{6g-6+2n}} = \frac{B(X)}{b_{g,n}} \cdot (I_\gamma)_*(\mu_{\text{Thu}}^\gamma)(A)
\]
for boxes
$
A := \prod_{i=1}^k [0,b_i) \subseteq (\mathbf{R}_{\geq 0})^k
$
with $\mathbf{b}:= (b_1,\dots,b_k) \in (\mathbf{R}_{>0})^k$ arbitrary. By Lemma \ref{lem:thu_meas_zero}, we can instead consider closed boxes
$
A := \prod_{i=1}^k [0,b_i] \subseteq (\mathbf{R}_{\geq 0})^k
$
with $\mathbf{b}:= (b_1,\dots,b_k) \in (\mathbf{R}_{>0})^k$ arbitrary.\\

For every $\mathbf{b}:= (b_1,\dots,b_k) \in (\mathbf{R}_{>0})^k$  and every $L > 0$ consider the counting function 
\[
f(X,\gamma,\mathbf{b},L) := \# \{\phi \in \text{Mod}_{g,n} \ | \ \ell_{\phi.\gamma_i}(X) \leq b_iL, \ \forall i =1,\dots,k \}.
\]
Notice that 
\[
f(X,\gamma,\mathbf{b},L) = \overline{\mu}_{\gamma,X}^L\left(\prod_{i=1}^k [0,b_i] \right).
\]
The proof of (\ref{eq:length_spec_filling_reduced}), and thus of Theorem \ref{theo:length_spec_filling}, reduces to the following result.\\

\begin{theorem}
	\label{theo:fil_count_comp}
	For every $\mathbf{b} := (b_1,\dots,b_k) \in (\mathbf{R}_{>0})^k$,
	\[
	\lim_{L \to \infty} \frac{f(X,\gamma,\mathbf{b},L)}{L^{6g-6+2n}} = \frac{B(X)}{b_{g,n}} \cdot \mu_{\text{Thu}}(\{\lambda \in \mathcal{ML}_{g,n} \ | \ i(\gamma_i,\lambda) \leq b_i, \ \forall i=1,\dots,k\}).
	\]
\end{theorem}
$ $

\textit{Generalizing Theorem \ref{theo:mir_count}.} As highlighted in \cite[\S 1.2]{Mir16}, Theorem \ref{theo:mir_count} for filling closed multi-curves holds for more general notions of length than total hyperbolic length. To prove Theorem \ref{theo:fil_count_comp} we make use of one such generalization, which we now describe.\\

Let $m \in \mathbf{N}$ be arbitrary. Every linear function $\mathcal{L} \colon \mathbf{R}^m \to \mathbf{R}$ cuts out a \textit{positive open half-space} and a \textit{positive closed half-space} in $\mathbf{R}^m$ corresponding to the sets
\begin{align*}
H_{>0}(\mathcal{L}) &:= \{x \in \mathbf{R}^m \ | \ \mathcal{L}(x) > 0\},\\
H_{\geq 0}(\mathcal{L}) &:= \{x \in \mathbf{R}^m \ | \ \mathcal{L}(x) \geq 0\}.
\end{align*}
A \textit{convex polytope} $P \subseteq \mathbf{R}^m$ is an intersection of finitely many positive open/closed half-spaces of $\mathbf{R}^m$. The boundary $\partial P \subseteq \mathbf{R}^m$ of a convex polytope $P \subseteq \mathbf{R}^m$ is its topological boundary when considered as a subset of $\mathbf{R}^m$.\\

Let $P \subseteq \mathbf{R}^m$ be a convex polytope. We say that a function $\mathcal{F} \colon P \to \mathbf{R}$ is \textit{asymptotically linear} if there exists a linear function $\mathcal{L} \colon P \to \mathbf{R}$ and a constant $c \in \mathbf{R}$ such that
\[
\lim_{x \in P \colon d(x,\partial P) \to \infty} \mathcal{F}(x) - \mathcal{L}(x) = c,
\]
where the distance $d$ corresponds to the standard Euclidean metric on $\mathbf{R}^m$. We say that a function $\mathcal{F} \colon P \to \mathbf{R}$ is \textit{asymptotically piecewise linear} if $P$ can be partitioned into finitely many convex polytopes on which $\mathcal{F}$ restricts to asymptotically linear functions.\\

The most important example for us of an asymptotically piecewise linear function is the hyperbolic length of a closed curve on $S_{g,n}$ as a function on Teichmüller space $\mathcal{T}_{g,n}$. See \cite[Theorem 4.1]{Mir16} for a proof of the following result.\\

\begin{theorem}
	\label{theo:length_APL}
	Let $\gamma$ be a closed curve on $S_{g,n}$. The hyperbolic length function 
	\[
	\ell_{\gamma} \colon \mathcal{T}_{g,n} \to \mathbf{R}_{>0}
	\]
	is asymptotically piecewise linear with respect to any set on Fenchel-Nielsen coordiantes $(\ell_i,\tau_i)_{i=1}^{3g-3+n}$. More concretely, after identifying 
	\[
	\mathcal{T}_{g,n} = \left( \mathbf{R}_{>0} \times \mathbf{R} \right)^{3g-3+n}
	\]
	using the Fenchel-Nielsen coordinates $(\ell_i,\tau_i)_{i=1}^{3g-3+n}$, the length function 
	\[
	\ell_{\gamma} \colon \mathcal{T}_{g,n} =  \left( \mathbf{R}_{>0} \times \mathbf{R} \right)^{3g-3+n} \to \mathbf{R}_{>0}
	\]
	is asymptotically piecewise linear.
\end{theorem}
$ $

Let $\mathcal{P} := (\gamma_1,\dots,\gamma_{3g-3+n})$ be a pair of pants decomposition of $S_{g,n}$ and $(\ell_i,\tau_i)_{i=1}^{3g-3+n}$ be a set of Fenchel-Nielsen coordinates of $\mathcal{T}_{g,n}$ adapted to $\mathcal{P}$. After identifying 
\[
\mathcal{T}_{g,n} = \left( \mathbf{R}_{>0} \times \mathbf{R} \right)^{3g-3+n}
\]
using the Fenchel-Nielsen coordinates $(\ell_i,\tau_i)_{i=1}^{3g-3+n}$, we can partition $\mathcal{T}_{g,n}$ into a countable union of convex polytopes of the form
\[
\mathcal{C}_{\mathcal{P}}^\mathbf{m} := \{ Y \in \mathcal{T}_{g,n} \ | \ m_i \cdot \ell_i(Y) \leq \tau_i(Y) < m_{i+1} \cdot \ell_{i+1}(Y) \}
\]
with $\mathbf{m} := (m_1,\dots,m_{3g-3+n}) \in \mathbf{Z}^{3g-3+n}$ is arbitrary. We say that a function $\mathcal{F} \colon \mathcal{T}_{g,n} \to \mathbf{R}_{>0}$ is \textit{bounding} with respect to the Fenchel-Nielsen coordinates $(\ell_i,\tau_i)_{i=1}^{3g-3+n}$ if for every $Y \in \mathcal{T}_{g,n}$ there exists a constant $C > 0$ such that for every $\mathbf{m}:= (m_1,\dots,m_{3g-3+n}) \in \mathbf{Z}^{3g-3+n}$ and every $Z \in\text{Mod}_{g,n} \cdot Y \cap \mathcal{C}_\mathcal{P}^\mathbf{m} \cap \mathcal{F}^{-1}([0,L])$,
\[
\ell_i(Z) \leq C \cdot \frac{L}{\max\{|m_i|,|m_{i}+1|\}}.
\]
$ $

The most important example for us of a bounding function is the total hyperbolic length of a filling closed multi-curve on $S_{g,n}$. See \cite[\S 9.4]{Mir16} for a proof of the following result.\\

\begin{proposition}
	\label{prop:len_bounding}
	Let $\gamma := (\gamma_1,\dots,\gamma_k)$ with $k \geq 1$ be an ordered filling closed multi-curve on $S_{g,n}$ and $\mathbf{a} := (a_1,\dots,a_k) \in (\mathbf{R}_{>0})^k$ be a vector of positive weights on the components of $\gamma$. The total hyperbolic length function
	\[
	\ell_{\mathbf{a} \cdot \gamma} \colon \mathcal{T}_{g,n} \to \mathbf{R}_{>0}
	\]
	is bounding with respect to any set of Fenchel-Nielsen coordinates on $ \mathcal{T}_{g,n}$.
\end{proposition}
$ $

Another important property of the total hyperbolic length function of a filling closed multi-curve on $S_{g,n}$ is its properness. See \cite[Lemma 3.1]{Ker83} for a proof of the following result. \\

\begin{lemma}
	\label{lem:proper_len}
	Let $\gamma := (\gamma_1,\dots,\gamma_k)$ with $k \geq 1$ be an ordered filling closed multi-curve on $S_{g,n}$ and $\mathbf{a} := (a_1,\dots,a_k) \in (\mathbf{R}_{>0})^k$ be a vector of positive weights on the components of $\gamma$. The total hyperbolic length function
	\[
	\ell_{\mathbf{a} \cdot \gamma} \colon \mathcal{T}_{g,n} \to \mathbf{R}_{>0}
	\]
	is proper.
\end{lemma}
$ $

Let $\mathcal{F} \colon \mathcal{T}_{g,n} \to \mathbf{R}_{>0}$ be a positive, continuous, proper function which is asymptotically piecewise linear and bounding with respect to some set of Fenchel-Nielsen coordinates. For every $L > 0$ consider the counting function
\[
f(X,\mathcal{F},L) := \#\{\phi \in \text{Mod}_{g,n} \ | \ \mathcal{F}(\phi \cdot X) \leq L\}.
\]
Consider also the limit
\begin{equation}
\label{eq:c(F)}
r(\mathcal{F}) := \lim_{L \to \infty} \frac{\mu_{\text{wp}}(\{ Y \in \mathcal{T}_{g,n} \ | \ \mathcal{F}(Y) \leq L\})}{L^{6g-6+2n}}.
\end{equation}
One can check this limit exists using Wolpert's magic formula and the properties of the function $\mathcal{F}$. \\

We are now ready to present the generalized version of the filling case of Theorem \ref{theo:mir_count} that we will need to prove Theorem \ref{theo:fil_count_comp}. Appropriately modifying the arguments in the proof of \cite[Theorem 1.1]{Mir16} yields the following result.\\

\begin{theorem}
	\label{theo:mir_closed_count_gen}
	Let $\mathcal{F} \colon \mathcal{T}_{g,n} \to \mathbf{R}_{>0}$ be a positive, continuous, proper function which is asymptotically piecewise linear and bounding with respect to some set of Fenchel-Nielsen coordinates. Then
	\[
	\lim_{L \to \infty} \frac{f(X, \mathcal{F},L)}{L^{6g-6+2n}} = \frac{B(X) \cdot r(\mathcal{F})}{b_{g,n}}.
	\]
\end{theorem}
$ $

\begin{remark}
	According to Theorem \ref{theo:length_APL}, Proposition \ref{prop:len_bounding}, and Lemma \ref{lem:proper_len}, given any set of positive rational weights $\mathbf{a}:=(a_1,\dots,a_k) \in (\mathbf{Q}_{>0})^k$ on the components of $\gamma$, the total hyperbolic length function $\ell_{\mathbf{a} \cdot\gamma} \colon \mathcal{T}_{g,n} \to \mathbf{R}_{>0}$ satisfies the hypothesis of  Theorem \ref{theo:mir_closed_count_gen}. It follows that we can recover the filling case of Theorem \ref{theo:mir_count} from Theorem \ref{theo:mir_closed_count_gen} by setting $\mathcal{F} := \ell_{\mathbf{a} \cdot \gamma}$.\\
\end{remark}

\textit{Topological interpretation of $r(\mathcal{F})$.} Following ideas similar to the ones introduced in the proof of Proposition \ref{prop:topo_dist_red}, one can give a topological interpretation of the limit $r(\mathcal{F})$ defined in (\ref{eq:c(F)}) for a particular class of maps $\mathcal{F} \colon \mathcal{T}_{g,n} \to \mathbf{R}_{>0}$, which we now describe.\\

Let $F \colon (\mathbf{R}_{\geq 0})^k \to \mathbf{R}_{\geq0}$ be a continuous, homogeneous, and proper function. In particular, $F$ is positive away from the origin. Consider the map $\mathcal{F} \colon \mathcal{T}_{g,n} \to \mathbf{R}_{>0}$ which to every $Y \in \mathcal{T}_{g,n}$ assigns the positive value
\begin{equation}
\label{eq:induced_F}
\mathcal{F}(Y) := F(\ell_{\gamma_1}(Y),\dots,\ell_{\gamma_k}(Y)).
\end{equation}
More generally, consider the map $\overline{\mathcal{F}} \colon \overline{\mathcal{Y}_{g,n}} \to \mathbf{R}_{>0}$ which to every $\alpha \in \overline{\mathcal{Y}_{g,n}}$ assigns the positive value
\begin{equation}
\label{eq:extended_F}
\overline{\mathcal{F}}(\alpha) := F(i(\gamma_1,\alpha),\dots,i(\gamma_k,\alpha)).
\end{equation}
For any map $\mathcal{F} \colon \mathcal{T}_{g,n} \to \mathbf{R}_{>0}$ as in (\ref{eq:induced_F}), the following result holds.\\

\begin{proposition}
	\label{prop:top_int_c}
	Let $\mathcal{F} \colon \mathcal{T}_{g,n} \to \mathbf{R}_{>0}$ be as in (\ref{eq:induced_F}). Then,
	\[
	r(\mathcal{F}) = \mu_{\text{Thu}}\left( \{ \lambda \in \mathcal{ML}_{g,n} \ | \ \overline{\mathcal{F}}(\lambda) \leq 1 \} \right),
	\]
	where $\overline{\mathcal{F}} \colon \overline{\mathcal{Y}_{g,n}} \to \mathbf{R}_{\geq0}$ is as in (\ref{eq:extended_F}).
\end{proposition}
$ $

\begin{proof}
	By Proposition \ref{prop:asymp_wp_meas},
	\[
	\lim_{t \to 0} \ t^{6g-6+2n} \cdot \mu_\text{wp}^t = \mu_{\text{Thu}}
	\]
	in the weak-$\star$ topology for measures on $\overline{\mathcal{Y}_{g,n}}$. Consider the subset of $\overline{\mathcal{Y}_{g,n}}$ given by
	\[
	D(\overline{\mathcal{F}}) := \{\alpha \in \overline{\mathcal{Y}_{g,n}} \ | \ \overline{\mathcal{F}}(\alpha) \leq 1 \}.
	\]
	Using the properness of the function $F \colon (\mathbf{R}_{\geq 0})^k \to \mathbf{R}_{\geq0}$ and the fact that $\gamma$ is filling, one can check that $D(\overline{\mathcal{F}}) \subseteq \overline{\mathcal{Y}_{g,n}}$ is compact. By Lemma \ref{lem:thu_meas_zero},
	\[
	\mu_{\text{Thu}}(\partial D(\overline{\mathcal{F}})) = 0.
	\]
	It follows from Portmanteau's theorem that
	\[
	\lim_{t \to 0} \ t^{6g-6+2n} \cdot \mu_\text{wp}^t(D(\overline{\mathcal{F}}) ) = \mu_{\text{Thu}}(D(\overline{\mathcal{F}})).
	\]
	Letting $t := 1/L$ for $L > 0$ and taking $L \to \infty$ we deduce
	\[
	\lim_{L \to \infty} \frac {\mu_\text{wp}^{1/L}(D(\overline{\mathcal{F}}))}{L^{6g-6+2n}} = \mu_{\text{Thu}}(D(\overline{\mathcal{F}})).
	\]
	As the function $F \colon (\mathbf{R}_{\geq 0})^k \to \mathbf{R}_{\geq0}$ is homogeneous,
	\[
	\mu_\text{wp}^{1/L}(D(\overline{\mathcal{F}})) = \mu_{\text{wp}}(\{Y \in \mathcal{T}_{g,n} \ | \ \mathcal{F}(Y) \leq L\})
	\]
	for every $L > 0$. Notice also that
	\[
	\mu_{\text{Thu}}(D(\overline{\mathcal{F}})) = \mu_{\text{Thu}}\left( \{ \lambda \in \mathcal{ML}_{g,n} \ | \ \overline{\mathcal{F}}(\lambda) \leq 1 \} \right).
	\]
	Putting things together we deduce
	\[
	r(\mathcal{F}) := \lim_{L \to \infty} \frac{\mu_{\text{wp}}(\{Y \in \mathcal{T}_{g,n} \ | \ \mathcal{F}(Y) \leq L\})}{L^{6g-6+2n}} = \mu_{\text{Thu}}\left( \{ \lambda \in \mathcal{ML}_{g,n} \ | \ \overline{\mathcal{F}}(\lambda) \leq 1 \} \right).
	\]
	This finishes the proof.
\end{proof}
$ $

\textit{Proof of Theorem \ref{theo:fil_count_comp}.} We are now ready to prove Theorem \ref{theo:fil_count_comp} and thus finish the proof of Theorem \ref{theo:length_spec_filling}.\\

\begin{proof}[Proof of Theorem \ref{theo:fil_count_comp}]
	Let $\mathbf{b} := (b_1,\dots,b_k) \in (\mathbf{R}_{>0})^k$ be arbitrary. Consider the function $F \colon (\mathbf{R}_{\geq 0})^k \to \mathbf{R}_{\geq 0}$ given by
	\[
	F(x_1,\dots,x_k) := \max \{x_1/b_1,\dots,x_k/b_k\}
	\]
	for every $(x_1,\dots,x_k) \in (\mathbf{R}_{\geq 0})^k$. Let $\mathcal{F} \colon \mathcal{T}_{g,n} \to \mathbf{R}_{>0}$ be the map induced by $F$ on $\mathcal{T}_{g,n}$ as defined in (\ref{eq:induced_F}). Notice that, as a consequence of Theorem \ref{theo:length_APL}, $\mathcal{F}$ is asymptotically piecewise linear with respect to any set of Fenchel-Nielsen coordinates. Let $\mathbf{1}:=(1,\dots,1) \in (\mathbf{Q}_{>0})^k$. The bound
	\[
	\ell_{\mathbf{1} \cdot \gamma} \leq k \cdot \max\{b_1,\dots,b_k\} \cdot \mathcal{F}
	\]
	together with Proposition \ref{prop:len_bounding} imply $\mathcal{F}$ is bounding with respect to any set of Fenchel-Nielsen coordinates. The same bound together with Lemma \ref{lem:proper_len} imply $\mathcal{F}$ is proper. By Theorem \ref{theo:mir_closed_count_gen} it follows that
	\[
	\lim_{L \to \infty} \frac{f(X,\mathcal{F},L)}{L^{6g-6+2n}} = \frac{B(X) \cdot r(\mathcal{F})}{b_{g,n}}.
	\]
	Notice that
	\[
	f(X,\mathcal{F},L)  = f(X,\gamma,\mathbf{b},L)
	\]
	for every $L > 0$. As a consequence of Proposition \ref{prop:top_int_c},
	\[
	r(\mathcal{F}) = \mu_{\text{Thu}}(\{\lambda \in \mathcal{ML}_{g,n} \ | \ i(\gamma_i,\lambda) \leq b_i, \ \forall i=1,\dots,k\}).
	\]
	Putting things together we deduce
	\[
	\lim_{L \to \infty} \frac{f(X,\gamma,\mathbf{b},L)}{L^{6g-6+2n}} = \frac{B(X)}{b_{g,n}} \cdot \mu_{\text{Thu}}(\{\lambda \in \mathcal{ML}_{g,n} \ | \ i(\gamma_i,\lambda) \leq b_i, \ \forall i=1,\dots,k\}).
	\]
	This finishes the proof.
\end{proof}
$ $


\bibliographystyle{amsalpha}


\bibliography{bibliography}

\providecommand{\bysame}{\leavevmode\hbox to3em{\hrulefill}\thinspace}
\providecommand{\MR}{\relax\ifhmode\unskip\space\fi MR }
\providecommand{\MRhref}[2]{%
  \href{http://www.ams.org/mathscinet-getitem?mr=#1}{#2}
}
\providecommand{\href}[2]{#2}
\begin{thebibliography}{{Ara}19b}

\bibitem[AA19]{AA19}
Francisco {Arana-Herrera} and Jayadev~S. {Athreya}, \emph{{Square-integrability
  of the Mirzakhani function and statistics of simple closed geodesics on
  hyperbolic surfaces}}, arXiv e-prints (2019), arXiv:1907.06287.

\bibitem[ABEM12]{ABEM12}
Jayadev Athreya, Alexander Bufetov, Alex Eskin, and Maryam Mirzakhani,
  \emph{Lattice point asymptotics and volume growth on {T}eichm\"uller space},
  Duke Math. J. \textbf{161} (2012), no.~6, 1055--1111. \MR{2913101}

\bibitem[{Ara}19a]{Ara20a}
Francisco {Arana-Herrera}, \emph{{Counting multi-geodesics on hyperbolic
  surfaces with respect to the lengths of individual components}}, In
  preparation, 2019.

\bibitem[{Ara}19b]{Ara19b}
\bysame, \emph{{Equidistribution of horospheres on moduli spaces of hyperbolic
  surfaces}}, arXiv e-prints (2019), arXiv:1912.03856.

\bibitem[Do13]{Do13}
Norman Do, \emph{Moduli spaces of hyperbolic surfaces and their
  {W}eil-{P}etersson volumes}, Handbook of moduli. {V}ol. {I}, Adv. Lect. Math.
  (ALM), vol.~24, Int. Press, Somerville, MA, 2013, pp.~217--258. \MR{3184165}

\bibitem[EM93]{EM93}
Alex Eskin and Curt McMullen, \emph{Mixing, counting, and equidistribution in
  {L}ie groups}, Duke Math. J. \textbf{71} (1993), no.~1, 181--209.
  \MR{1230290}

\bibitem[EM18]{EM18}
Viveka {Erlandsson} and Gabriele {Mondello}, \emph{{Ergodic invariant measures
  on the space of geodesic currents}}, arXiv e-prints (2018), arXiv:1807.02144.

\bibitem[EMM19]{EMM19}
Alex {Eskin}, Maryam {Mirzakhani}, and Amir {Mohammadi}, \emph{{Effective
  counting of simple closed geodesics on hyperbolic surfaces}}, arXiv e-prints
  (2019), arXiv:1905.04435.

\bibitem[EPS16]{EPS16}
V.~{Erlandsson}, H.~{Parlier}, and J.~{Souto}, \emph{{Counting curves, and the
  stable length of currents}}, ArXiv e-prints (2016).

\bibitem[ES16]{ES16}
Viveka Erlandsson and Juan Souto, \emph{Counting curves in hyperbolic
  surfaces}, Geom. Funct. Anal. \textbf{26} (2016), no.~3, 729--777.
  \MR{3540452}

\bibitem[ES20]{ES20}
Viveka {Erlandsson} and Juan {Souto}, \emph{Geodesic currents and
  {M}irzakhani's curve counting}, In preparation, 2020.

\bibitem[EU18]{EU18}
V.~{Erlandsson} and C.~{Uyanik}, \emph{{Length functions on currents and
  applications to dynamics and counting}}, ArXiv e-prints (2018).

\bibitem[FLP12]{FLP12}
Albert Fathi, Fran\c{c}ois Laudenbach, and Valentin Po\'enaru, \emph{Thurston's
  work on surfaces}, Mathematical Notes, vol.~48, Princeton University Press,
  Princeton, NJ, 2012, Translated from the 1979 French original by Djun M. Kim
  and Dan Margalit. \MR{3053012}

\bibitem[FM12]{FM11}
Benson Farb and Dan Margalit, \emph{A primer on mapping class groups},
  Princeton Mathematical Series, vol.~49, Princeton University Press,
  Princeton, NJ, 2012. \MR{2850125}

\bibitem[Ker83]{Ker83}
Steven~P. Kerckhoff, \emph{The {N}ielsen realization problem}, Ann. of Math.
  (2) \textbf{117} (1983), no.~2, 235--265. \MR{690845}

\bibitem[KM96]{KM96}
D.~Y. Kleinbock and G.~A. Margulis, \emph{Bounded orbits of nonquasiunipotent
  flows on homogeneous spaces}, Sina\u{\i}'s {M}oscow {S}eminar on {D}ynamical
  {S}ystems, Amer. Math. Soc. Transl. Ser. 2, vol. 171, Amer. Math. Soc.,
  Providence, RI, 1996, pp.~141--172. \MR{1359098}

\bibitem[{Liu}19]{Liu19}
Mingkun {Liu}, \emph{{Length statistics of random multicurves on closed
  hyperbolic surfaces}}, arXiv e-prints (2019), arXiv:1912.11155.

\bibitem[Mar04]{Mar04}
Grigoriy~A. Margulis, \emph{On some aspects of the theory of {A}nosov systems},
  Springer Monographs in Mathematics, Springer-Verlag, Berlin, 2004, With a
  survey by Richard Sharp: Periodic orbits of hyperbolic flows, Translated from
  the Russian by Valentina Vladimirovna Szulikowska. \MR{2035655}

\bibitem[{Mar}16]{Mar16}
Bruno {Martelli}, \emph{{An Introduction to Geometric Topology}}, ArXiv
  e-prints (2016).

\bibitem[Mas85]{Mas85}
Howard Masur, \emph{Ergodic actions of the mapping class group}, Proc. Amer.
  Math. Soc. \textbf{94} (1985), no.~3, 455--459. \MR{787893}

\bibitem[Mir04]{Mir04}
Maryam Mirzakhani, \emph{Simple geodesics on hyperbolic surfaces and the volume
  of the moduli space of curves}, ProQuest LLC, Ann Arbor, MI, 2004, Thesis
  (Ph.D.)--Harvard University. \MR{2705986}

\bibitem[Mir07a]{Mir07b}
\bysame, \emph{Random hyperbolic surfaces and measured laminations}, In the
  tradition of {A}hlfors-{B}ers. {IV}, Contemp. Math., vol. 432, Amer. Math.
  Soc., Providence, RI, 2007, pp.~179--198. \MR{2342816}

\bibitem[Mir07b]{Mir07a}
\bysame, \emph{Simple geodesics and {W}eil-{P}etersson volumes of moduli spaces
  of bordered {R}iemann surfaces}, Invent. Math. \textbf{167} (2007), no.~1,
  179--222. \MR{2264808}

\bibitem[Mir07c]{Mir07c}
\bysame, \emph{Weil-{P}etersson volumes and intersection theory on the moduli
  space of curves}, J. Amer. Math. Soc. \textbf{20} (2007), no.~1, 1--23.
  \MR{2257394}

\bibitem[Mir08a]{Mir08a}
\bysame, \emph{Ergodic theory of the earthquake flow}, Int. Math. Res. Not.
  IMRN (2008), no.~3, Art. ID rnm116, 39. \MR{2416997}

\bibitem[Mir08b]{Mir08b}
\bysame, \emph{Growth of the number of simple closed geodesics on hyperbolic
  surfaces}, Ann. of Math. (2) \textbf{168} (2008), no.~1, 97--125.
  \MR{2415399}

\bibitem[{Mir}16]{Mir16}
M.~{Mirzakhani}, \emph{{Counting Mapping Class group orbits on hyperbolic
  surfaces}}, ArXiv e-prints (2016).

\bibitem[Pap88]{Pap88}
Athanase Papadopoulos, \emph{Sur le bord de {T}hurston de l'espace de
  {T}eichm\"{u}ller d'une surface non compacte}, Math. Ann. \textbf{282}
  (1988), no.~3, 353--359. \MR{967017}

\bibitem[Pap91]{Pap91}
\bysame, \emph{On {T}hurston's boundary of {T}eichm\"{u}ller space and the
  extension of earthquakes}, Topology Appl. \textbf{41} (1991), no.~3,
  147--177. \MR{1135095}

\bibitem[PH92]{PH92}
R.~C. Penner and J.~L. Harer, \emph{Combinatorics of train tracks}, Annals of
  Mathematics Studies, vol. 125, Princeton University Press, Princeton, NJ,
  1992. \MR{1144770}

\bibitem[PP93]{PP93}
A.~Papadopoulos and R.~C. Penner, \emph{The {W}eil-{P}etersson symplectic
  structure at {T}hurston's boundary}, Trans. Amer. Math. Soc. \textbf{335}
  (1993), no.~2, 891--904. \MR{1089420}

\bibitem[PS15]{Pap15}
A.~Papadopoulos and W.~Su, \emph{On the {F}insler structure of
  {T}eichm\"{u}ller's metric and {T}hurston's metric}, Expo. Math. \textbf{33}
  (2015), no.~1, 30--47. \MR{3310926}

\bibitem[RS19]{RS19}
Kasra Rafi and Juan Souto, \emph{Geodesic currents and counting problems},
  Geom. Funct. Anal. \textbf{29} (2019), no.~3, 871--889. \MR{3962881}

\bibitem[RS20]{RS20}
Kasra {Rafi} and Juan {Souto}, \emph{Statistics of simple curves on surfaces,
  revisited}, In preparation, 2020.

\bibitem[SB01]{BS01}
Ya\c{s}ar S\"{o}zen and Francis Bonahon, \emph{The {W}eil-{P}etersson and
  {T}hurston symplectic forms}, Duke Math. J. \textbf{108} (2001), no.~3,
  581--597. \MR{1838662}

\bibitem[{Thu}98]{Thu98}
William~P. {Thurston}, \emph{{Minimal stretch maps between hyperbolic
  surfaces}}, arXiv Mathematics e-prints (1998), math/9801039.

\bibitem[{Wri}19]{Wri19}
Alex {Wright}, \emph{{A tour through Mirzakhani's work on moduli spaces of
  Riemann surfaces}}, arXiv e-prints (2019), arXiv:1905.01753.

\end{thebibliography}

$ $

\end{document}